\documentclass[]{amsart}

\usepackage[utf8]{inputenc}
\usepackage[T1]{fontenc}

\usepackage{lineno,hyperref, graphics, graphicx, xcolor, frcursive,comment,booktabs}
\usepackage{amsmath}
\usepackage{latexsym}
\usepackage{amsthm}
\usepackage{amssymb, ifthen, enumerate}
\usepackage{longtable}

\usepackage{color}
\usepackage{ulem}

\newtheorem{lemma}{Lemma}[section]
\newtheorem{thm}[lemma]{Theorem}
\newtheorem{prop}[lemma]{Proposition}

\newtheorem{rem}[lemma]{Remark}
\newtheorem{hyp}[lemma]{Hypothesis}

\newtheorem{ex}[lemma]{Example}

\newcommand{\Z}{\mathbbm{Z}}	
\newcommand{\C}{\mathbbm{C}}	
\newcommand{\Alt}{\mathrm{Alt}}     
\newcommand{\Sym}{\mathrm{Sym}}     

\newcommand{\aut}{\textrm{Aut}}

\newcommand{\fix}{\textrm{fix}}
\newcommand{\id}{\textrm{id}}
\newcommand{\GL}{\textrm{GL}}
\newcommand{\SL}{\textrm{SL}}

\newcommand{\Gal}{\textrm{Gal}}

\newcommand{\PSL}{\textrm{PSL}}
\newcommand{\PGL}{\textrm{PGL}}
\newcommand{\PSU}{\textrm{PSU}}

\newcommand{\GF}{\textrm{GF}}
\newcommand{\Sz}{\textrm{Sz}}
\newcommand{\GLi}{\textrm{GL}}
\newcommand{\POm}{P\Omega}
\newcommand{\SLi}{\textrm{SL}}
\newcommand{\PSp}{\textrm{PSp}}

\newcommand{\M}{\textrm{M}}

\newcommand{\syl}{\textrm{Syl}}
\newcommand{\Syl}{\textrm{Syl}}

\newcommand{\FO}{\textrm{fix}_{\Omega}}

\newcommand{\Inn}{\textrm{Inn}}

\def \<{\langle }
\def \>{\rangle }
\def \N{\mathbb{N}}

\def \Z{\mathbb{Z}}

\def \C{\mathbb{C}}

\makeatletter 

\begin{document}

\begin{center}
\Large{\textbf{Finite permutation groups that act with fixity $4$}}

\vspace{0.2cm} \small{Barbara Baumeister, Paula H\"ahndel, Kay Magaard, Christoph M\"oller, and Rebecca Waldecker}
\end{center}

\normalsize
\vspace{2cm}

Abstract:

Motivated by the theory of Riemann surfaces and specifically the significance of Weierstrass points, we prove general structure results about finite groups that have a faithful transitive action with fixity $4$. We also explain examples for many different possibilities of such actions. 
\vspace{1cm}

Keywords:

Permutation groups, fixity, fixed points

\vspace{1cm}


\section{Introduction}

In this article, we prove general results about permutation groups that act transitively, faithfully, and with fixity 4.
We were originally motivated by Riemann surfaces and their automorphism groups, but we also found the corresponding questions about permutation groups interesting in their own right. In fact, there is a long history of research on permutation groups that act with constraints. For example, Bender worked on transitive permutation groups where involutions have no fixed point or exactly one fixed point. This led to the notion of strongly embedded subgroups (see \cite{Bender}), and we frequently encounter strongly embedded subgroups in our work. Hering extended Bender's hypothesis to involutions with at most two fixed points (\cite{Hering}), and Ronse, who to our knowledge coined the term "fixity" (see \cite{Ronse1980}), proved results about groups where involutions fix at most 15 points (\cite{Ronse}). 

Throughout this article, we say that a group \textbf{$G$ acts with fixity $k \in \N$ on a set $\Omega$} if and only if
the maximum number of fixed points of elements in $G^\#$ is exactly $k$.
Faithful transitive permutation groups of fixity $1$ are Frobenius groups, where the Frobenius kernel consists of the fixed point free elements and the Frobenius complements are the point stabilisers. In previous work, we have already analysed groups of fixity 2 and 3 in detail (\cite{MW},\cite{HW}, \cite{MW3} and \cite{MW3corr}), as well as finite simple groups that act with fixity 4 (see \cite{BHMSW}). This is the foundation of our general analysis of groups that act with fixity 4.

As usual, we denote by $F(G)$ the \textbf{Fitting subgroup} of $G$, that is the largest nilpotent normal subgroup of $G$. The \textbf{components} of $G$
are the quasi-simple subnormal subgroups of $G$, and the product of all the components of $G$ is the layer $E(G)$. The subgroup $F^*(G) = F(G)E(G)$ is the \textbf{generalized Fitting subgroup} of $G$,  see  for instance \cite[Section~31]{Aschbacher}.
\bigskip

\begin{thm}\label{FinalTheorem}
    Suppose that $G$ is a finite group that acts faithfully, transitively, and with fixity
		$4$
		on a set $\Omega$.
		Let $\alpha \in \Omega$.
	Then one of the following is true, and all cases occur:
	\begin{enumerate}[\quad (1)]
		\item $F(G)\cap G_{\alpha} \neq 1$ and $E(G)=1$. Furthermore, one of the following holds:
		\begin{enumerate}[(a)]
			\item $F(G)$ is a $2$-group with sectional $2$-rank at most $4$ that acts faithfully and with fixity $4$ on $\alpha^{F(G)}$, and the point stabilisers have order dividing $2^3 \cdot 3^2$.
			\item $G$ is one of the groups in the group column in Table~\ref{TableNotFaithful} and $|\Omega|\leq 28$. Moreover, $F(G)$ is either elementary abelian of order $4$, $8$, $9$ or $16$, or it has structure $C_4\times C_4$ or $(C_4 \times C_4) : C_2$.
		\end{enumerate}
			\item $F(G)$ acts semi-regularly on $\Omega$. Furthermore, for all $p\in \pi(G_{\alpha})$, the Sylow $p$-subgroups of $G$ have $p$-rank $1$ or $F(G)=O_2(G) \times O_3(G)$.
			Additionally, we have one of the following possibilities:
			\begin{enumerate}[(a)]
				\item $E(G)=1$.%
                
				\item  $E(G)$ acts semi-regularly on $\Omega$ and $E(G)/Z(E(G))\cong \Alt_5$.
                If $E(G) \cong \Alt_5$, then $G_{\alpha}$ has isomorphism type $C_2$, $E_4$, $C_4$, $\Sym_3$, or $\Alt_4$.
				If $E(G) \cong \SLi_2(5)$, then $G_{\alpha}$ has isomorphism type $C_2$, $E_4$, or $C_4$.%
				
				\item $E(G)$ acts with fixity $2$ on $\alpha^{E(G)}$, $E(G) \cong \PSL_2(q), Sz(q)$ or $\PSL_3(4)$ for some prime power $q$ and one of the following is true:
    
				\begin{enumerate}[(i)]
					\item $F(G) = 1$ and $G$ is almost simple.
					\item $F(G) \cong C_2$ and $G = A \times F(G)$, where $A$ is an almost simple group which acts with fixity $2$ on both of its orbits.
					\item $F(G) \cong C_2$ and $E(G)$ is isomorphic to $\Alt_5$ or to $\PSL_2(7)$. 
				\end{enumerate}

				\item $E(G)$ acts with fixity $3$ on $\alpha^{E(G)}$, $E(G) \cong \Alt_6$, $F(G)=1$, and one of the following holds:
				
				\begin{enumerate}[(i)]
					\item $|\Omega|=6$ and $G\cong \Sym_6$.
					\item $|\Omega|=12$ and $G$ is isomorphic to $\M_{10}$, $\PGL_2(9)$, or $\aut(\Alt_6)$.
				\end{enumerate}%
                
				\item $E(G)$ acts with fixity $4$ on $\alpha^{E(G)}$ and additionally one of the following holds:
				\begin{enumerate}[(i)]
					\item $|F(G)| = 1$ and $E(G)$ is isomorphic to one of the groups appearing in Table~\ref{TableAllFix} with a fixity $4$ action.
					\item $|F(G)| = 2$ and $E(G) \cong \PSL_2(q)$, where $q \geq 7$ is an odd prime power.
					If $q \equiv 1$ modulo $4$, then $E(G) \cap G_\alpha$ is isomorphic to $C_{\frac{q-1}{4}}$ or $E_q : C_{\frac{q-1}{4}}$.
					If $q \equiv -1$ modulo $4$, then $E(G) \cap G_\alpha$ is isomorphic to $C_{\frac{q+1}{4}}$.
					\item $|F(G)| = 2$ and $E(G)$ is isomorphic to $\SLi_2(q)$, where $q$ is an odd prime power, to $C_2.\Sz(8)$, or to $C_2.\PSL_3(4)$.
					\item $|F(G)| = 4$ and $G$ is isomorphic to $\SL_2(5) : C_2$ with $G_\alpha$ isomorphic to $\Sym_3$ or $D_{10}$, or $G$ is isomorphic to $\GL_2(5)$ with $G_\alpha \cong C_5 : C_4$.
				\end{enumerate}
			\end{enumerate}
		\end{enumerate}
	\end{thm}

Throughout this article we will give more details for some of the cases, for example in Lemma \ref{EtransFix2}, where we extend Case (c)~(iii). The following little table is meant to guide towards examples for specific cases of our main result:

\begin{table}[ht]
	\centering
		\begin{tabular}{l|l ||l|l ||l|l}
Case & Reference &Case & Reference & Case & Reference\\
	\hline \hline
   $ (1) (a)$ & Example~\ref{ExampleFix4}~~ & $(2) (c) (i)$ & Example~\ref{ExFinalTheorem2c}& $(2) (e) (i)$ & \cite{BHMSW}\\
    $ (1) (b)$ & Example~\ref{Ex1(b)}~~ & $(2) (c) (ii)$ & Lemma~\ref{Fix2toFix4}& $(2) (e) (ii)$ & Example~\ref{ExFinalTheorem2eii}\\
   $ (2) (a)$ & Example~\ref{ExFinalTheorem2a}~~ & $(2) (c) (iii)$ & Example~\ref{EtransFix2eg}& $(2) (e) (iii)$ & Example~\ref{2.PSL_2.Sz}, Lemma~\ref{SLquasisimple} \\
     $ (2) (b)$ & Example~\ref{exRegComp}~~ & $(2) (d) $ & Example~\ref{Alt6Fix3}& $(2) (e) (iv)$ & Example~\ref{ExFinalTheorem2eiv} \\
 	\end{tabular}
    	\vspace{0.2cm}
	\caption{Assignment of the examples to the different cases of Theorem~\ref{FinalTheorem}}
	\label{ListOfExamples}
    \end{table}

We introduce some notation for the upcoming tables:
If $n \in \N$ and $p \in \N$ is a prime, then we 
denote the cyclic group of order $n$, the dihedral group of order $n$ or the semi-dihedral group of order $n$ by
$C_n$, $D_n$ or $SD_n$, respectively, and we denote the elementary abelian group of order $p^n$ by $E_{p^n}$. If $A$ and $B$ are groups, then we denote the semi-direct product of $A$ and $B$, where $A$ is normal, by $A:B$. $A \times B$ denote the direct product, $A.B$ and $A \cdot B$ denote unspecified products and $H=A^{\cdot} B$ denotes a non-split extension where $A \unlhd H$ and $G/A \cong B$.

\newpage

\begin{table}
	\centering

	\begin{tabular}{l||l|l|l|l}
	Row&	Small Group ID & group $G$ & $F(G)$ & $G_{\alpha}$\\
		\hline \hline
	1&	$[  24,   12]$&
		$\Sym_4$ & $E_4$ & $C_2$\\ \hline
	2&	$[  24,   13]$
		&$C_2 \times \Alt_4$& $E_8$& $C_2$, $E_4$\\ \hline
	3&	$[  48,    3]$
		&$(C_4 \times C_4) : C_3$ & $C_4 \times C_4$ & $C_4$\\ \hline
	4&	$[  48,   30]$
		&${C_2}^{\cdot}\Sym_4$ & $E_8$ & $C_4$\\ \hline
	5&	$[  48,   48]$
		&$C_2 \times \Sym_4$ & $E_8$& $C_4$, $E_4$, $D_8$\\ \hline
	6&	$[  48,   49]$
		&$E_4 \times \Alt_4$ & $E_{16}$ & $E_4$\\ \hline
	7&	$[  48,   50]$
		&$(E_4\times E_4):C_3$ & $E_{16}$ & $E_4$\\ \hline
	8&	$[  56,   11]$
		&$E_8 : C_7$ & $E_8$ & $C_2$\\ \hline
	9&	$[  72,   39]$
		&$E_9 : C_8$ & $E_9$ & $\Sym_3$\\ \hline
	10&	$[  72,   40]$
		&$\Sym_3 \wr C_2$ & $E_9$ & $D_{12}$\\ \hline
	10&	$[  72,   41]$
		&$E_9 : Q_8$ & $E_9$ & $\Sym_3$\\ \hline
	12&	$[  80,   49]$
		&$E_{16} : C_5$ & $E_{16}$ & $E_4$\\ \hline
	13&	$[  96,   64]$
		&$(C_4\times C_4):\Sym_3$ & $C_4 \times C_4$& $C_8$, $C_4 \times C_2$, $D_8$, $Q_8$\\ \hline
	14&	$[  96,   72]$
		&${E_8}^{\cdot}\Alt_4$ & $(C_4 \times C_4) : C_2$ & $D_8$\\ \hline
	15&	$[  96,  195]$
		&${C_2}^{\cdot} (C_2\times \Sym_4)$ & $E_{16}$ & $D_8$\\ \hline
	16&	$[  96,  227]$
		&$(E_4\times E_4):\Sym_3$ & $E_{16}$ & $C_4\times C_2$, $D_8$, $E_8$, $\Alt_4$\\ \hline
	17&	$[ 144,  182]$
		&$E_9 : SD_{16}$ & $E_9$ & $D_{12}$\\ \hline
	18&	$[ 144,  184]$
		&$\Alt_4 \times \Alt_4$ & $E_{16}$ & $\Alt_4$\\ \hline
	19&	$[ 160,  234]$
		&$E_{16}:D_{10}$ & $E_{16}$ & $C_4 \times C_2$, $E_8$\\ \hline
	20&	$[ 168,   43]$
		&$E_8 : (C_7 : C_3)$ & $E_8$ & $C_6$\\ \hline
	21&	$[ 192,  956]$
		&${E_8}^{\cdot}\Sym_4$
		 & $(C_4 \times C_4) : C_2$ & $D_{16}$, $SD_{16}$\\ \hline
	22&	$[ 216,  153]$
		&$E_9:\SLi_2(3)$ %
		& $E_9$& $C_3 \times \Sym_3$\\ \hline
	23&	$[ 288, 1025]$
		&$\Alt_4\wr C_2$ &$E_{16}$ & $C_2 \times \Alt_4$\\ \hline
	24&	$[ 288, 1026]$
		&$(\Alt_4 \times \Alt_4) : C_2$ & $E_{16}$ & $\Sym_4$\\ \hline
	25&	$[ 320, 1635]$
		& $E_{16}:(C_5:C_4)$ & $E_{16}$ & %
		$E_4:C_4$, $C_8:C_2$\\ \hline
	26&	$[ 432,  734]$
		& $E_9:\GLi_2(3)$
		& $E_9$ & $\Sym_3 \times \Sym_3$\\ \hline
	27&	$[ 960,11357]$
		&$(E_{4}\times E_4):\SLi_2(4)$ & $E_{16}$ & $(C_4 \times C_4) : C_3$,\\ %
		&&&& $(E_4\times E_4): C_3$\\ \hline
	28&	$[1344,  814]$
		&${E_8}^{\cdot} \PSL_3(2)$ & $E_8$ & $\GLi_2(3)$  
	\end{tabular}
	\vspace{0.2cm}
	\caption{Groups acting transitively, faithfully, and with fixity $4$, and such that $F(G)$ does not act faithfully on $\alpha^{F(G)}$}
	\label{TableNotFaithful}
\end{table}

\newpage
\begin{table}
	\centering

\begin{tabular}{c||c|c|c}
& \multicolumn{3}{c}{Point stabiliser structure} \\
Group $G$ & Fixity $2$ & Fixity $3$ & Fixity $4$ \\ \hline \hline
$\Alt_5$ & $C_2, C_3, C_5,$ & $E_4$ &  \\
& $\Sym_3, D_{10}, \Alt_4$ & & \\ \hline
$\Alt_6 \cong \PSL_2(9)$ & $C_4, C_5, E_9:C_4$ & $\Sym_4, \Alt_5$ & $C_2, \Sym_3, E_9,$\\
& & & $D_{10}, E_9:C_2$ \\ \hline
$\Alt_7$ &  & $C_7, \PSL_2(7)$ & $C_5, \Alt_6$ \\ \hline
$\Alt_8$ &  & $C_7$ &  \\ \hline
$\PSL_2(7) \cong \PSL_3(2)$ & $C_3, C_4, C_7:C_3, \Alt_4$ & $\Sym_4$ & $C_2, \Sym_3$ \\ \hline
$\PSL_2(8)$ & $C_7, C_9, E_8:C_7$ &  & $C_2, \Sym_3, D_{14}, D_{18}$ \\ \hline
$\PSL_2(11)$ & $C_5, C_6, C_{11}:C_5$ & $\Alt_5$ & $C_3, \Alt_4$ \\ \hline
$\PSL_2(13)$ & $C_6, C_7, C_{13}:C_6$ &  & $C_3, \Alt_4, C_{13} : C_3$ \\ \hline
$\PSL_2(q), q \geq 16, q \equiv 1$ mod $4$ & $C_{\frac{q-1}{2}}, C_{\frac{q+1}{2}}, E_q:C_{\frac{q-1}{2}}$ &  & $C_{\frac{q-1}{4}}, E_q:C_{\frac{q-1}{4}}$ \\ \hline
$\PSL_2(q), q \geq 16, q \equiv -1$ mod $4$ & $C_{\frac{q-1}{2}}, C_{\frac{q+1}{2}}, E_q:C_{\frac{q-1}{2}}$ &  & $C_{\frac{q+1}{4}}$ \\ \hline
$\PSL_2(q), q \geq 16, q \equiv 0$ mod $2$ & $C_{q-1}, C_{q+1},E_q:C_{q-1}$ &  &  \\ \hline
$\PSL_3(4)$ & $C_5$ & $C_7$ &  \\ \hline
$\PSL_3(q), q \geq 3$ &  & $C_{\frac{q^2+q+1}{(3,q-1)}}$ &  \\ \hline
$\PSL_4(3)$ &  & $C_{13}$ &  \\ \hline
$\PSL_4(5)$ &  & $C_{31}$ &  \\ \hline
$\PSU_3(3)$ &  & $C_7$ & $((C_3 \times C_3) : C_3 ) : C_8$ \\ \hline
$\PSU_3(q), q > 3$ &  & $C_{\frac{q^2-q+1}{(3,q+1)}}$ &  \\ \hline
$\PSU_4(3)$ &  & $C_7$ & $C_5$ \\ \hline
$\PSp_4(q), q \geq 3$ &  &  & $C_{\frac{q^2 + 1}{(2,q+1)}}$ \\ \hline
$\Sz(q)$ & Sylow $2$ normaliser or $C_{q-1}$ &  & $C_{q + \sqrt{2q} + 1}$, $C_{q - \sqrt{2q} + 1}$ \\ \hline
$\POm_8^-(q)$ &  &  & $C_{\frac{q^4 + 1}{(2,q+1)}}$ \\ \hline
$^3D_4(q)$ &  &  & $C_{q^4-q^2+1}$ \\ \hline
$^2G_2(q)$ &  &  & order $q^3 \cdot \frac{q-1}{2}$ or $\frac{q-1}{2}$ \\ \hline
$M_{11}$ &  & $\Alt_6~^{\cdot}C_2$ & $C_5, C_{11}:C_5, \PSL_2(11)$ \\ \hline
$M_{12}$ &  &  & $M_{11}$ \\ \hline
$M_{22}$ &  & $C_7$ & $C_5, C_{11}:C_5$ \\ \hline
$J_1$ &  &  & $C_{15}$

\end{tabular}
\vspace{0.2cm}
	\caption{Finite simple groups acting transitively, faithfully, and
		with fixity $2$, $3$, or $4$.
        The first column shows the simple group $G$. The second, third and fourth columns list the structures of possible point stabilisers for fixity $2$, $3$ and $4$ actions, respectively.
        } \label{TableAllFix}
\end{table}

After collecting some preliminary results in Section 2, we will discuss examples in Section 3. Then we turn to the analysis of the Fitting subgroup in Section 4 and, similarly, we investigate the layer and the possibilities for components in Section 5. This is where most of the work towards our main results happens. 
After that, we collect everything and prove the theorem.


\newpage
\section{Preliminaries}

In this article, all groups are supposed to be finite. We use standard notation for orbits and point stabilisers, and whenever $\Omega$ is a set and $X$ is a group, or $x$ is a group element, then we denote the fixed point set of $X$, or $x$, on $\Omega$ by $\FO(X)$ or $\FO(x)$, respectively.

Additional notation will be introduced when we need it. 
\newpage

\subsection{Some basic properties of groups of fixity $k$.}

\begin{lemma}\label{normaliser}

Suppose that $G$ is a finite transitive permutation
group with permutation domain $\Omega$. Suppose further that
 $G$ acts with fixity~$k \in \N$ on $\Omega$ and let $\alpha \in \Omega$.

(i) If $1 \neq X \le G_{\alpha}$, then $|N_G(X):N_{G_\alpha}(X)| \le k$.
In particular, if
$\FO(X)=\{\alpha\}$, then
$N_G(X) \le G_\alpha$.

(ii) If $x \in G_{\alpha}^\#$, then $|C_G(x):C_{G_\alpha}(x)| \le k$. In particular, if
$\FO(x)=\{\alpha\}$, then
$C_G(x) \le G_\alpha$.

(iii) If $k=2$, then $G$ has even order.

(iv) $|Z(G)|$ divides $k$.

(v) If $p \in \pi(G)$ and $p>k$, then either $G_\alpha$ is a $p'$-group or it contains a Sylow $p$-subgroup of $G$.

(vi) $Z(G)$ acts semi-regularly on $\Omega$.

(vii) If $k=4$ and $G_\alpha$ has order coprime to $6$, then $G_\alpha$ is a Frobenius group or it is a four point stabiliser.
\end{lemma}

\begin{proof}
This is mainly Lemma 2.2 in \cite{HW}. Statement (vi) follows from the fact that all subgroups of $Z(G)$ are normal in $G$ and that $G$ acts faithfully and transitively on $\Omega$ by hypothesis, which implies that $Z(G)$ intersects all point stabilisers trivially.
Finally, (vii) is Lemma 8.2~(d) of \cite{BHMSW}.
\end{proof}

\begin{lemma}\label{numberfpt}
Suppose that $G$ is a finite group, that $1\neq U\lneq G$ and that $x\in G^\#$.
Then, in the action of $G$ on the coset space $G/U$ by right multiplication, the number of fixed points of
$x$ is exactly
\[\frac{|\{\langle x \rangle^g \leq U\mid g\in G\}|\cdot |N_G(\langle x \rangle)|}{|U|} = \frac{|x^G \cap U| \cdot |C_G(x)|}{|U|} \,.\]
\end{lemma}

\begin{proof}
The first part is Lemma 4.2 in \cite{BHMSW}. To obtain the second part, we note that the number of fixed points can be obtained as values of the corresponding permutation character. This character comes from inducing the principle character of $U$ to $G$ (see Lemma 5.14 in \cite{Isaacs}). Thus, we obtain
\[
    |\FO(x)| = \frac{ | \{ g \in G \mid x^g \in U \} | }{|U|} = \frac{ \sum_{y \in x^G \cap U}  |C_G(y)| }{|U|} = \frac{|x^G \cap U| \cdot |C_G(x)|}{|U|},
\]
which is the second part of the statement.

\end{proof}

\begin{rem}\label{CompInd}

\begin{enumerate}
    \item[(a)]
    Suppose that $E$ is a quasi-simple group. Then $E$ does not contain any proper subgroup of index at most $4$: Assume otherwise and let $H$ be such a subgroup. Then the action of $E$ on $E/H$ by right multiplication leads to a proper quotient of $E$ that is isomorphic to a non-trivial subgroup of the soluble group $\Sym_4$. This contradicts the fact that $E$ is quasi-simple. The same applies for central products of quasi-simple groups. In particular, if $G$ is a finite group, then $E(G)$ does not contain a proper subgroup of index at most $4$.

	If $G$ acts transitively and faithfully with fixity $4$ on a set $\Omega$, and if $\alpha \in \Omega$ and $E := E(G) \neq 1$, then $E \neq E_\alpha$. This is because the action is faithful and the previous paragraph implies that $\alpha^E = |E:E_\alpha| \geq 5$. In particular, since every non-trivial element in $G$ has at most four fixed points, the action of $E$ on $\alpha^E$ is faithful.

\item[(b)]	If a group $G$ acts transitively and faithfully with fixity $1$ on a set $\Omega$, then $G$ is a Frobenius group. That means that there exists a proper non-trivial normal subgroup $K \trianglelefteq G$ and that, for all $x \in G \setminus K$, we have that $C_K(x) = 1$ \cite[8.1.12]{KS}. In particular, $Z(G) = 1$.
	If, in addition, $G$ is quasi-simple, then $Z(G) = 1$ forces $G$ to be simple, which contradicts the existence of $K$.
\end{enumerate}
\end{rem}

\begin{lemma}\label{centrereduction_pre}
Let $G$ be a finite group and suppose that $G$ acts transitively and faithfully with fixity $k \leq 4$ on a set $\Omega$. Let $Z \leq Z(G)$ and suppose that $|\Omega| \leq k \cdot |Z|$. Then $G$ is soluble.
\end{lemma}
\begin{proof}
Let $\bar{\Omega}$ denote the set of $Z$-orbits on $\Omega$. By Lemma \ref{normaliser} (iv) and (vi), $|Z|$ divides $k$ and $Z$ acts semi-regularly on $\Omega$. Since $Z$ is normal in $G$, we may look at the action of $G$ on $\bar{\Omega}$. Let $K$ denote the kernel of this action.

Since $|\bar{\Omega}| \leq k \leq 4$, $G/K$ is isomorphic to a subgroup of $\Sym_4$, and hence it is soluble. Let $x \in K$ be of prime order $p$. If $p \geq 5$, then $p > |Z|$ and $x$ fixes every $Z$-orbit pointwise, in contradiction to the faithful action of $G$. Therefore, $K$ is a $\{2,3\}$-group and hence soluble. It follows that $G$ is soluble.
\end{proof}

\begin{lemma}\label{centrereduction}
Let $k \in \N$ and suppose that the finite group $G$ acts transitively and with fixity $k$

on a set $\Omega$. Let $Z \le Z(G)$ and suppose that $G$ is not soluble or that $|\Omega| >k \cdot |Z|$.

Then $\bar{G} := G/Z$ acts transitively and non-regularly on the set $\bar{\Omega}$ of $Z$-orbits. Moreover, if we let $\omega \in \Omega$ and $\bar{\omega} := \omega^Z$, then the following holds:
\begin{enumerate}[\quad (i)]
\item $|G_\omega| = |\bar{G}_{\bar{\omega}}|$ and the full preimage of $\bar{G}_{\bar{\omega}}$ in $G$ is the direct product $Z \times G_\omega$.
\item $\bar{G}$ acts faithfully and with fixity at most $k$ on $\bar{\Omega}$.
\item If $k \leq 4$ and $H \leq G_\omega$ is such that $|\fix_{\bar{\Omega}}(\bar{H})| > k/|Z|$, then $|\bar{H}|$ divides $k$.
\end{enumerate}
\end{lemma}
\begin{proof}
If $G$ is not soluble, then Lemma \ref{centrereduction_pre} implies that $|\Omega| > k \cdot |Z|$. Hence, we suppose for the remainder of the proof that this inequality holds.

Lemma 3.7 of \cite{HW} states that $\bar{G}$ acts transitively and non-regularly with fixity at most $k$ on $\bar{\Omega}$ and that $|\bar{G}_{\bar{\omega}}| = |G_\omega|$. The action is faithful because $|\bar \Omega| = |\Omega| / |Z| > k$. The full preimage $U$ of $\bar{G}_{\bar{\omega}}$ in $G$ contains $G_\omega$ and, since $Z$ permutes the set of fixed points of $G_\omega$, it follows that $G_\omega$ is the kernel of the action of $U$ on $\bar{\omega}$. Now $G_\omega$ and $Z$ are normal subgroups of $U$ and $G_\omega \cap Z = 1$, which implies that $U = Z \times G_\omega$. This already proves (i) and (ii).

For the last part of the statement, let $k \leq 4$ and $H \leq G_\omega$ be such that $\ell := |\fix_{\bar{\Omega}}(\bar{H})| > k / |Z|$. For all $i \in \{1,...,\ell\}$ let $\omega_i \in \Omega$ and $\bar{\omega_i} := \omega_i^Z \in \bar{\Omega}$ be such that $\fix_{\bar{\Omega}}(\bar{H}) = \{\bar{\omega_1},...,\bar{\omega_\ell}\}$. Furthermore, let $K_i \leq ZH$ denote the kernel of $ZH$ in the action on $\bar{\omega_i}$. Since $Z \leq Z(ZH)$ permutes the set of fixed points of every element in $ZH$ semi-regularly, by Lemma \ref{normaliser}, it follows that $ZH/K_i$ acts regularly on $\bar{\omega_i}$. This means that $|ZH:K_i| = |Z|$. Every $x \in \bigcap_i K_i$ fixes the set $\bar{\omega_i}$ point-wise, which implies that it fixes at least $\ell \cdot |Z| > k$ points on $\Omega$. Then our fixity $k$ hypothesis gives that $x = 1$ and therefore $\bigcap_i K_i = 1$.

If $x \in H$ is such that $o(x)$ is prime and does not divide  $|Z|$, then $x \in K_i$ for every $i \in \{1,...,\ell\}$. But this contradicts the fact that $\bigcap_i K_i = 1$. Hence $H$ is a $2$-group in the case where $k \in \{2,4\}$, and it is a $3$-group in the case where $k = 3$. If $|Z| = k$ and $|H| > k$, then $|ZH:K_1| = |ZH:K_2| = |Z|$ implies that $K_1 \cap K_2 \neq 1$, but every element in $K_1 \cap K_2$ fixes $2 \cdot |Z| > k$ points on $\Omega$. This contradicts our fixity $k$ hypothesis. Hence the statement holds if $|Z| = k$.\\
In the remaining case, where $k = 4$ and $|Z| = 2$, the fixity $4$ hypothesis implies that $K_1 \cap K_2 \cap K_3 = 1$. If $|H| \geq 8$, then $|ZH| \geq 16$, and then we arrive at the contradiction $|K_1 \cap K_2 \cap K_3| \geq 16/2^3 \geq 2$. Hence $|H| \leq 4$, which concludes the proof.
\end{proof}

\begin{rem}
    The hypothesis that $G$ is not soluble or that $|\Omega| >k \cdot |Z|$ might seem unusual. But it turned out that, in our applications, we actually need both alternatives of this hypothesis, and it was the easiest and most flexible version to state it in this way.  
\end{rem}

\subsection{Construction principles for groups of given fixity.}

We need the notion of strongly closed subgroups and hence recall it briefly. If $U \leq G$, then we say that a subgroup $H \leq U$ is \textbf{strongly closed in $U$ with respect to $G$} if  and only if $H^g \cap U \leq H$ for all $g \in G$.\\
The following lemma will be used later to create examples.

\begin{lemma}\label{centreextension}
Let $G$ be a finite group and $Z \leq Z(G)$. Let $H \leq G$ be such that $H \cap Z = 1$ and suppose that $\bar{G} := G / Z$ acts with fixity $k \in \N$ on $\bar{G}/\bar{H}$ via right multiplication. If $H$ is strongly closed in $ZH$ with respect to $G$, then $G$ acts with fixity $|Z| \cdot k$ on $G/H$ via right multiplication.
\end{lemma}
\begin{proof}
We are going to show that $|\fix_{G/H}(x)| = |\fix_{\bar{G}/\bar{H}}(\bar{x})| \cdot |Z|$ for every $x \in H^\#$. According to Lemma \ref{numberfpt}, we have the following: 
\[
	|\fix_{G/H}(x)| = \frac{ |\{ \<x\>^g \leq H \mid g \in G \}| \cdot |N_G(\<x\>)| }{|H|} \quad
	\text{and}
	\quad
	|\fix_{\bar{G}/\bar{H}}(\bar{x})| = \frac{ |\{ \<\bar{x}\>^g \leq \bar{H} \mid g \in G \}| \cdot |N_{\bar{G}}(\<\bar{x}\>)| }{|\bar{H}|}.
\]
Our hypothesis $Z \cap H = 1$ yields that $|H| = |\bar{H}|$. For every $g \in G$, the strong closed-ness of $H$ in $ZH$ implies that $\<x\>^g \leq H$ whenever $\<x\>^g \leq ZH$. Hence $|\{\<\bar{x}\>^g \leq \bar{H} \mid g \in G\}| = |\{ \<x\>^g \leq H \mid g \in G \}|$. Let $g \in G$ be such that $\bar{g} \in N_{\bar{G}}(\<\bar{x}\>)$. Then $\<x\>^g \leq Z \<x\>$ and $\<x\>^g \leq H$, because $H$ is strongly closed in $ZH$. In particular, $\<x\>^g \leq Z \<x\> \cap H = \<x\>$ by Dedekind's identity and therefore $g \in N_G(\<x\>)$. Hence $|N_G(x)| = |Z| \cdot |N_{\bar{G}}(\bar{x})|$ and 
Lemma \ref{numberfpt} gives that 
$|\fix_{G/H}(x)| = |\fix_{\bar{G}/\bar{H}}(\bar{x})| \cdot |Z|$.
\end{proof}

\begin{lemma}\label{StabInNsg}
	Let $G$ be a finite group and $N \trianglelefteq G$. Suppose that $N$ acts transitively on a set $\Omega$ and let $\alpha \in \Omega$. Let $n := |G:N|$ and $\{ g_1,...,g_n \}$ be a right transversal of $N$ in $G$. Then $G$ acts with fixity
	\[
		\max_{x \in N^\#} \sum_{i = 1}^{n} | \FO(x^{g_i}) |
	\]
	on $G / N_\alpha$ by right multiplication.
\end{lemma}
\begin{proof}
	Note that for every $x \in N$, $|\FO(x)|$ is the value of the permutation character $\chi_{\text{perm}}$ for the action of $N$ on $\Omega$. By Lemma 5.14 in \cite{Isaacs}, this character is obtained by inducing the principle character of $N_\alpha$ to $N$. Then, by Exercise 5.1 in \cite{Isaacs}, we obtain the permutation character $\chi_{\text{perm}}^G$ for the action of $G$ on $G / N_\alpha$ by inducing $\chi_{\text{perm}}$ to $G$. We see that $\chi_{\text{perm}}^G$ vanishes on $G \setminus N$. Then the definition of induced characters yields that $\chi_{\text{perm}}^G (x) = \sum_{i = 1}^n |\FO(x^{g_i})|$ for all $x \in N$. Finally, the definition of fixity concludes the proof.
\end{proof}

The following lemma shows how, given two positive integers
$r$ and $k$ and a group that acts with fixity $k$, we can construct a group that acts with
fixity $r \cdot k$. 

\bigskip
\begin{lemma}\label{Fix2toFix4}
Let $k,r \in \N$ and let $G$ be a group with the following properties:
$E \le G$ acts faithfully, transitively,
		and with fixity $k$ on a set
		$\Delta$, $c \in G$ has order~$r$ and $G= E \times \langle c \rangle$.
		Then $G$ acts faithfully, transitively, and with fixity $r\cdot k$ on a set of size $r\cdot
		|\Delta|$.
\end{lemma}
\begin{proof}

    Let $H$ be a point stabiliser of $E$ in its action on $\Delta$ and note that $\<c\> \leq Z(G)$. We see that $E \trianglelefteq G$ and that $\{ c^1,...,c^r \}$ is a right transversal for $E$ in $G$. By Lemma \ref{StabInNsg}, $G$ acts with fixity
	\[
		\max_{x \in E^\#} \sum_{i = 1}^r |\fix_{\Delta}(x^{c^i})| = \max_{x \in E^\#} \sum_{i = 1}^r |\fix_{\Delta}(x)| = r \cdot \max_{x \in E^\#} |\fix_{\Delta}(x)| = r \cdot k
	\]
	on $G / H$ by right multiplication. This action is faithful and transitive. Moreover, $|E:H|=|\Delta|$ by hypothesis and therefore $|G/H|=\frac{|G|}{|H|}=\frac{r \cdot |E|}{|H|}=r \cdot |\Delta|$ as claimed.
\end{proof}

\begin{lemma}\label{RegNSG}
	Let $K$ be a group, let $1 \neq H \le \aut(K)$ and set $k := \max \{|C_K(h)| \mid h \in H^\# \}$.
	 Then $G := K:H$ acts transitively and with fixity $k$ on $\Omega := G / H$ by right multiplication. The point stabilisers are isomorphic to $H$ and $K \trianglelefteq G$ acts regularly on $\Omega$.
\end{lemma}
\begin{proof}
	Clearly $G$ acts transitively on $\Omega$ with point stabilisers isomorphic to $H$. By construction, $K \trianglelefteq G$ and $K \cap H = 1$. Since every point stabiliser is conjugate to $H$ in $G$, we see that $K$ acts semi-regularly. Moreover, $|\Omega| = |G:H| = |K|$, which means that the action is in fact regular. Now (15.11) in \cite{Aschbacher} allows us to identify the action of $H$ on $\Omega$ with the conjugation action of $H$ on $K$. Thus, $k$ is the maximum number of fixed points of non-trivial elements in $G$.
\end{proof}

Finally, we need one more concept before we look at orbits of normal subgroups.
We say that two permutation groups $(G_1, \Omega_1)$ and $(G_2, \Omega_2)$
are \textbf{isomorphic as permutation groups} if and only if there are a bijection $f: \Omega_1 \rightarrow \Omega_2$ and a group isomorphism $\psi: G_1 \rightarrow G_2$ such that 
$$f(\omega^g) = (f(\omega))^{g^\psi}$$
for all $\omega \in \Omega_1$ and $g \in G_1$.

\begin{lemma}\label{normalaction}
Suppose that the group $G$ acts transitively and faithfully on a set $\Omega$. Let $N \unlhd G$ and $\alpha \in \Omega$.
Then, the permutation groups $(N,\alpha^N)$ and $(N, \beta^N)$ are isomorphic  for all $\beta \in \Omega$. 
\end{lemma}

\begin{proof}
Let $\beta \in \Omega$. Since $G$ acts transitively on $\Omega$, there exists an element $g \in G$ such
that $\alpha^g = \beta$. Let $\psi$ be the automorphism of  $N$ that is induced by conjugation by $g$. Then $N_\alpha^\psi = N_\beta$. Let $f:\alpha^N \rightarrow \beta^N$ be defined as follows: If $h \in N$, then $f(\alpha^h):=\beta^{h^\psi}$. This is well-defined by the previous paragraph. If $h \in N$ and $\omega:= \alpha^h \in \alpha^N$, then it follows for all $x \in N$ that 
$$f(\omega^x)= f(\alpha^{hx}) = \beta^{(hx)^\psi} = \beta^{h^\psi x^\psi} =
f(\alpha^h)^{x^\psi} = f(\omega)^{x^\psi}.$$
Therefore, the permutation groups $(N,\alpha^N)$ and $(N, \beta^N)$
are isomorphic.
\end{proof}

\subsection{\texttt{GAP} code.}

There are several places where we use \texttt{GAP} to look at individual cases or construct specific examples. 
Therefore, we show two functions that we frequently use when \texttt{GAP} supports our arguments.

The first function takes an integer $k$ and returns a function which checks whether a given transitive permutation group $G$ acts with fixity $k$.

\begin{verbatim}
HasFixity := k -> function(G)
    local g, exact;
    if NrMovedPoints(G) <= k then
        return false;
    fi;
    if not IsTrivial( Stabiliser( G, MovedPoints(G){[1..(k+1)]}, OnTuples) ) then
        return false;
    fi;
    exact := false;
    for g in List(ConjugacyClasses(G), Representative) do
        if g = Identity(G) then continue; fi;
        if NrMovedPoints(G) - NrMovedPoints(g) > k then return false; fi;
        if NrMovedPoints(G) - NrMovedPoints(g) = k then exact := true; fi;
    od;
    return exact;
end;
\end{verbatim}

The next function takes a transitive permutation group and returns its fixity.
\begin{verbatim}
GetFixity := G -> Maximum(List(
    Difference(ConjugacyClasses(G), [Identity(G)^G] ),
    x -> NrMovedPoints(G) - NrMovedPoints(Representative(x))
));
\end{verbatim}

\section{Examples}

\begin{ex}\label{ExampleFix4}
	Consider the ring of Eisenstein integers $R := \Z[\omega]$, where $\omega \in \C$ is a primitive third root of unity. Let $n \geq 2$ and $V := R / 2^n R \cong C_{2^n} \times C_{2^n}$ and let $\varphi \in \aut(V)$ be the automorphism of $V$ (as a group with addition) that is induced by multiplication by $\omega + 1$.
    Then $o(\varphi) = 6$, $\varphi^2$ corresponds to multiplication by $(\omega +1)^2=\omega$ and $\varphi^3$ corresponds to multiplication by $(\omega +1)^3=-1$.     
    A small calculation shows that $C_V(\varphi^2) = 0$ and that $C_V(\varphi^3) \cong C_2 \times C_2$. Let $G := V:H$, where $H := \< \varphi \>$ and where we view $V$ as $C_{2^n} \times C_{2^n}$ with multiplication. By Lemma \ref{RegNSG}, 
    $G$ acts with fixity $4$ on $\Omega := G / H$ by right multiplication. Moreover, $F(G) \cong (C_{2^n} \times C_{2^n}) : C_2$ and $F(G) \cap H = \<\varphi^3\> \cong C_2$ contains an involution with exactly four fixed points. Thus we have constructed an infinite family of examples for Case (1)(a) of Theorem \ref{FinalTheorem} where $G$ is not a $2$-group.
\end{ex}

For the remainder of this article, we will often need the following set-up:

\begin{hyp}\label{hyp}
  Suppose that $G$ is a finite group that acts faithfully, transitively and with fixity
		$4$
		on a set $\Omega$.
		Let $\alpha \in \Omega$.
\end{hyp}        

\subsection{Examples where $F(G)$ does not act semi-regularly.}

We start by describing pairs $(G,G_\alpha)$ where $G$ satisfies Hypothesis~\ref{hyp} in its action on a set $\Omega$, and where in addition $|\Omega| \leq 28$ and $F(G)$ does not act faithfully on $\alpha^{F(G)}$. These are the examples presented in Table~\ref{TableNotFaithful}.

\begin{ex}\label{Ex1(b)}
Let $F$ be a field of prime power order $q$ and let $H \leq \GL_2(q)$. Denote by $V$ the natural $H$-module and consider the group $G := V: H$. We can define an action of $G$ on $\Omega := \{ U + v \mid v \in V$$,  $ $U \leq V$ and $U$ is 1-dimensional$ \}$ in the following way: For all $x \in V$, $\varphi \in H$ and $U + v \in \Omega$ let $(U + v)^{\varphi x} := U^\varphi + v^\varphi + x$. This action is transitive if $H$ is transitive on the set of $1$-dimensional subspaces of $V$.

For every $0 \neq x \in V$, we see that $\fix_\Omega(x) = \{ \< x \>_F + v \mid v \in V \}$ has size $q$. Hence $G$ acts with fixity at least $q$ on $\Omega$. 
Let $x \in V$ and $\varphi \in H^\#$ be such that $\varphi x$ has more than $q$ fixed points on $\Omega$. Then there exist $U_1 + v_1, U_2 + v_2 \in \FO(\varphi x)$, where $U_1 \neq U_2$. For each $i \in \{1,2\}$, let $0 \neq u_i \in U_i$ and note that $U_i^\varphi = U_i$ implies the existence of $\lambda_i \in F$ such that $u_i^\varphi = \lambda_i u_i$. Since $\{u_1, u_2\}$ is a basis of $V$, we find $\alpha_i, \beta_i \in F$ such that $U_i + v_i = U_i + \alpha_i u_{3 - i}$ and $x = \beta_1 u_1 + \beta_2 u_2$. Now
\[
	(U_i + v_i)^{\varphi x} = (U_i + \alpha_i u_{3 - i})^{\varphi x} = U_i + (\alpha_i \lambda_{3 - i} + \beta_{3 - i} ) u_{3-i} = U_i + \alpha_i u_{3 - i}
\]
and thus $\alpha_i = \alpha_i \lambda_{3 - i} + \beta_{3 - i}$. If follows that $\alpha_1$ and $\alpha_2$ are uniquely determined by $\varphi x$ unless $\lambda_{3 - i} = 1$, in which case $\beta_{3 - i} = 0$.
If $\lambda_1 = \lambda_2 = 1$, then $\varphi$ is the identity in contradiction to $\varphi \in H^\#$. If $\lambda_1 = \lambda_2 \neq 1$, then $\varphi$ stabilises every $1$-dimensional subspace and for every $1 \neq U \leq V$ there exists exactly one $U + v \in V/U$ that is fixed by $\varphi x$. This makes a total of $q+1$ fixed points.
If $\lambda_1 \neq \lambda_2$, then $\varphi$ has exactly two distinct eigenspaces and thus every fixed point of $\varphi x$ can be written as $U_1 + w$ or $U_2 + w$ for some $w \in V$. In that case, $\varphi x$ has two fixed points if $\lambda_1 \neq 1 \neq \lambda_2$ and $q+1$ fixed points if $\lambda_1 = 1$ or $\lambda_2 = 1$.

We conclude that $G$ acts transitively with fixity $q+1$ if $H$ contains a non-trivial multiple of the identity or a diagonalisable matrix with eigenvalue $1$ of multiplicity $1$, and with fixity $q$ otherwise. In this way, we obtain all the examples in Table \ref{TableNotFaithful} where $F(G) \cong E_9$ and some of the examples where $F(G) \cong E_{16} \cong \GF(4)^2$. 
\end{ex}

\begin{ex}\label{ExRow8and20}
	Similarly to the previous example, we can take a field $F$ of order $8$ and its multiplicative group, denoted by $F^\times$, and we can construct an action of a subgroup $G$ of $F:F^\times : \Gal(F)$, where $G$ contains $F:F^\times$, on the set $\Omega := \{ U + v \mid v \in F, U \leq F$ of order $2 \}$ in the following way: For all $U + v \in \Omega$, $x \in F$, $m \in F^\times$ and $\varphi \in \Gal(F)$ we define $(U + v)^{\varphi m x} := m U^\varphi + m v^\varphi + x$. A small calculation shows that this is a transitive fixity $4$ action. In this way, we obtain Rows 8 and 20 of Table \ref{TableNotFaithful}.
\end{ex}

\begin{ex}\label{ExRow28}
According to page 46 of \cite{Atlas}, the group $\PSp_6(2)$ admits a permutation representation on a set $\Omega$ with 28 points. This representation is also used by the \texttt{GAP} package AtlasReps \cite{AtlasRep} for \texttt{AtlasGroup("S6(2)")}. In $\PSp_6(2)$ we find a maximal subgroup of isomorphism type $E_{2^6} : \PSL_3(2)$ which contains a subgroup $G \cong {E_8}^\cdot \PSL_3(2)$. Using \texttt{GAP}, we see that $G$ acts transitively with fixity $4$ on such a set $\Omega$ of size 28 with point stabilisers isomorphic to $\GL_2(3)$. This is Row 28 of Table \ref{TableNotFaithful}.
\end{ex}

\subsection{Examples where $F(G)$ acts semi-regularly.}

In the next few examples we look more closely at
Case (2) of Theorem~\ref{FinalTheorem}, and then we will come back to it in Section 5. 

\begin{ex}\label{ExFinalTheorem2a}
	Let $A := \<a\>$ be a cyclic group of order $4$ and let $B := \< b \>$ be a non-trivial cyclic group of odd order. Define $K := A \times B$ and consider the involution $\varphi \in \aut(K)$ that acts as follows: $a^\varphi := a$ and $b^\varphi := b^{-1}$. Then $C_A(\varphi) = A$, $C_B(\varphi) = 1$ and thus $C_K(\varphi) = A$ has order $4$. By Lemma \ref{RegNSG},
    $G := K:\< \varphi \>$ acts with fixity $4$ on $\Omega := G / \<\varphi\>$ by right multiplication and $K = F(G)$ is a regular normal subgroup. Thus, $G$ is an example for Case (2)\,(a) in Theorem~\ref{FinalTheorem} in its action on $\Omega$ as defined above. 
\end{ex}

\bigskip
\begin{ex}\label{ExFinalTheorem2c}
    Let $G := \Sym_5$ and $H := \< (1,2,3) \>$. Then $H \leq E := E(G) = \Alt_5$ and a small calculation shows that $N_G(H) = \< (1,2,3), (1,2), (3,4) \> \cong \Sym_3 \times C_2$ and $N_E(H) = \< (1,2,3), (1,2)(4,5) \> \cong \Sym_3$.
    Consider the action of $G$ on $\Omega := G / H$ by right multiplication. By Lemma \ref{numberfpt}, the non-trivial elements in $H$ have $|N_G(H)|/|H| = 4$ fixed points on $\Omega$ and $|N_E(H)|/|H| = 2$ fixed points on every $E$-orbit. Thus $G$ acts with fixity $4$ on $\Omega$, while $E$ acts with fixity $2$ on each of its orbits. This is an example for Case (2)\,(c)\,(i) of Theorem \ref{FinalTheorem}.
\end{ex}

\bigskip
\begin{ex}\label{Alt6Fix3}
Case (2)\,(d)\,(i) corresponds to the natural action of $\Sym_6$ on $\{1,...,6\}$, and we see that this action has fixity $4$. For the other case, let $G$ be one of the groups $M_{10}$, $\PGL_2(9)$ or $\aut(\Alt_6)$ and let $N \trianglelefteq G$ be the normal subgroup of index $2$ which is isomorphic to $\Alt_6$ or $\Sym_6$, respectively. Consider the natural action of $N$ on $\{1,...,6\}$. In all cases, there exists an involution $t \in G \setminus N$ which interchanges subsets of $N$ as follows, by conjugation: 
$3$-cycles in $N$ are mapped to double-$3$-cycles in $N$, and vice versa, and transpositions in $N$ are mapped to triple transpositions in $N$, and vice versa.
Let $H \leq N$ be a point stabiliser in the natural action on $\{1,...,6\}$. We let $x \in N^\#$ and we are interested in the value $|\fix_{N/H}(x)| + |\fix_{N/H}(x^t)|$, which we will now calculate case by case: 

If $x \in N^\#$ is a transposition or a triple transposition, then $|\fix_{N/H}(x)| + |\fix_{N/H}(x^t)|=4+0 = 4$.

If $x \in N^\#$ is a double transposition or a $4$-cycle, then
the value is $2 + 2 = 4$.

If $x$ is a $3$-cycle or a double-$3$-cycle, then the value is $3 + 0 = 3$.

If $x$ is a $5$-cycle or the product of a transposition with a disjoint $3$-cycle, then the value is $1+1 = 2$, and in all other cases, the value is $0$.

Since $\{ \id, t\}$ is a transversal of $N$ in $G$, Lemma \ref{StabInNsg} gives that $G$ acts with fixity $4$ on $G/H$. We also note that $\Alt_6 \cong E(G) \leq N$ acts naturally and therefore with fixity $3$ on one of its orbits, and hence we have constructed examples for Case (2)\,(d)\,(ii).
\end{ex}

\bigskip
\begin{ex}\label{ExFinalTheorem2eii}
	Let $q \geq 5$ be an odd prime power, $E := \PSL_2(q)$, let $t$ be an involution and let $G := E \times \< t \>$. Let $C := \< c \> \leq E$ be a cyclic subgroup chosen as follows: 
    If $q \equiv 1$ modulo $4$, then $o(c)=(q-1)/2$ and if $q \equiv -1$ modulo $4$, then $o(c)=(q+1)/2$. Set $H := \< c t \> \cong C$ and consider the action of $G$ on $\Omega := G / H$ by right multiplication. Since $|G:E| = |H:H \cap E| = |H:\<c^2\>| = 2$, we know that $E$ acts transitively on $\Omega$, and this is a fixity $4$ action by Table \ref{TableAllFix}. Moreover, $|H|$ is divisible by $4$, all elements of prime order in $H$ are contained in $E$, and fixity $4$ is attained by some element of prime order. This implies that all elements in $G^\#$ fix at most four points, and hence $G$ acts with fixity $4$ on $\Omega$ as well. 
    This is Case (2)\,(e)\,(ii) of Theorem \ref{FinalTheorem}.
	
\end{ex}

\bigskip
\begin{ex}\label{ExFinalTheorem2eiv}
	Let $\SLi_2(5) \leq G \leq \GL_2(5)$ and consider the natural action of $G$ on a two-dimensional vector space $V$ over a field with five elements. Then $G$ acts transitively and faithfully on $\Omega := V \setminus \{ 0_V \}$. Moreover, every non-trivial element in $G$ centralises either a zero-dimensional or a one-dimensional subspace of $V$, which means that it has zero or four fixed points in $\Omega$. Thus, the action has fixity $4$, and the elements of order $5$ have exactly four fixed points. If $Z(\GL_2(5)) \leq G$, then $G$ is an example for Case (2)\,(e)\,(iv) of Theorem \ref{FinalTheorem} with point stabilisers of isomorphism type $D_{10}$ or $C_5 : C_4$.

	\medskip

    For the remaining case of (2)\,(e)\,(iv), we use the Transitive Groups Library~\cite{TransGrp} and look at the group \texttt{G := TransitiveGroup(40, 188)}. The group has isomorphism type $\SL_2(5) : C_2$ and the point stabilisers are isomorphic to $\Sym_3$. The expression \texttt{GetFixity(G)} in \texttt{GAP} evaluates to $4$.
\end{ex}

There are some more examples for Case (e) of Theorem \ref{FinalTheorem} in Section 5.

\section{The structure and action of the Fitting subgroup}

We begin with a general result about nilpotent groups that act with fixity at most $4$, and we will see that the centre plays an important role.

\bigskip
\begin{lemma}\label{nilpotent4}
    Let $k$ be a positive integer and let $G$ be a nilpotent group that acts transitively, faithfully, and with fixity $k$ on a set $\Omega$ such that $1\leq k\leq 4$. Let $\alpha \in \Omega$.
	Then one of the following holds:
    \begin{enumerate}[\quad (1)]
        \item $k=2$, $G$ is a dihedral or semi-dihedral $2$-group, and $|G_\alpha| = 2$.
        \item $k=3$, $G$ is a $3$-group of maximal class, and $|G_\alpha| = 3$.
        \item $k=4$, $G$ is a $2$-group of sectional $2$-rank at most $4$, and $|G_\alpha| \leq 8$.
	\end{enumerate}
\end{lemma}

\begin{proof}
	Since $G$ is nilpotent, $Z(G)$ is non-trivial and $|Z(G)|$ is divisible by every prime
	divisor of $|G|$. %
	By Lemma~\ref{normaliser}~(iv), %
	$|Z(G)|$ divides $k$, in particular $k>1$.
	
	First suppose that $k=2$. Then $Z(G)$ is a $2$-group and therefore $G$ is a $2$-group. Since there exists a non-trivial element in $G$ with exactly
	two fixed points, $|G:G_{\alpha}|=|\Omega|\geq 4$. Then by Lemma~3.2
	in~\cite{HW}, $G$ is
	dihedral or semi-dihedral and $G_\alpha$ has order 2 or index 2.
    If we assume for a contradiction that $|G_\alpha| \neq 2$, then 
    $2=|G:G_{\alpha}|=|\Omega|$, contrary to the fact that $G$ acts faithfully on $\Omega$.

	Next suppose that $k=3$. Then $Z(G)$, and hence $G$, is a $3$-group. In particular, both $|G_{\alpha}|$ and $|\Omega|=|G:G_{\alpha}|$ are
	divisible by $3$. Even more, since $G$ contains a non-trivial element with three fixed points, we see that $|\Omega|>3$. Now
	$|\Omega|$ is divisible by $9$ and therefore $|G| \ge 27$.
    This means that Case (3) of 	
	Lemma~3.2 in~\cite{MW3corr} holds, hence $G$ has maximal class. 
	If $|G_{\alpha}|$ is divisible by $9$, then $0 \equiv |\Omega|\equiv |\fix_{\Omega}(G_{\alpha})| =3\mod 9$, which is impossible.
	As a consequence, $|G_{\alpha}|=3$ as stated.
	
	Finally, suppose that $k=4$.
	Then $G$ is a $2$-group by Lemma~\ref{normaliser} (iv), and there exists an involution $t\in G$ with exactly four fixed points. The fact that $G$ is a $2$-group also implies that $G$ does not have a strongly embedded
	subgroup, which together with the main theorem
	in~\cite{Ronse} shows that $G$ has sectional $2$-rank at most $4$. Moreover, $G$ acts faithfully and with fixity $4$, which forces $|\Omega|\ge 6$ and hence $|G : G_\alpha| > 4$. Then Lemma 9 in \cite{SylowFix4} gives that $|G_\alpha| \leq 8$.
\end{proof}

We distinguish cases depending on whether the Fitting subgroup  acts semi-regularly or not.
We start with the situation where the Fitting subgroup does not act semi-regularly, which will turn out to be the case that needs most work. 

\subsection{$F(G)$ not acting semi-regularly.}

\begin{lemma}\label{FfixedE1} 
    Suppose that Hypothesis \ref{hyp} holds.
	Suppose that $F(G)\cap G_{\alpha}\neq 1$. Then $E(G)=1$.
\end{lemma}
\begin{proof}
	Let $b\in F(G)\cap G_{\alpha}$ be non-trivial.
	Since $E(G)\leq C_G(b)$, Lemma~\ref{normaliser}~(ii) implies
	$|E(G):E(G) \cap G_\alpha| \leq |C_G(b):C_{G_{\alpha}}(b)| \leq 4$. Thus by Remark~\ref{CompInd}~(a),
	$E(G)\leq G_{\alpha}$ and the faithful and transitive action of~$G$ implies $E(G)=1$.
\end{proof} %

\bigskip
\noindent
Further we will distinguish between whether the Fitting subgroups acts faithfully on its orbits or not.

The next lemma show that faithful action can be restated in terms of the centre of the Fitting subgroup.

\bigskip
\begin{lemma}\label{Ffaithful}
	Suppose that Hypothesis \ref{hyp} holds. Then $Z(F(G)) \cap G_\alpha = 1$ if and only if $F(G)$ acts faithfully on $\alpha^{F(G)}$.
\end{lemma}
\begin{proof}
	Consider the action of $F := F(G)$ on $\alpha^F$. If $Z(F) \cap F_\alpha \neq 1$, then $Z(F) \cap F_\alpha$ is a non-trivial normal subgroup of $F$ which is contained in $F_\alpha$. Therefore it is contained in $F_\omega$ for all $\omega \in \alpha^F$, and it follows that $F$ does not act faithfully on $\alpha^F$. Conversely, if $F$ does not act faithfully on $\alpha^F$, then the kernel $K$ of this action is a non-trivial normal subgroup of $F$. Then $Z(F) \cap K \neq 1$, because $F$ is nilpotent. Thus $1 \neq Z(F) \cap K \leq Z(F) \cap G_\alpha$.
\end{proof}

\bigskip\noindent
We start with the case where the Fitting subgroup acts faithfully on its orbits. It turns out that this is only possible if the Fitting subgroup itself acts with fixity $4$.

\bigskip
\begin{lemma}\label{FixityFG}
	Suppose that Hypothesis \ref{hyp} holds, that $F(G) \cap G_\alpha \neq 1$ and that $F(G)$ acts faithfully on $\alpha^{F(G)}$. Then $F(G)$ acts with fixity $4$ on $\alpha^{F(G)}$.
	In particular, $F(G)$ is a $2$-group of sectional $2$-rank at most $4$.
\end{lemma}
\begin{proof}
	Let $F:=F(G)$. By Lemma~\ref{FfixedE1}, $E(G)=1$.
	Thus by Theorem 6.5.8 in \cite{KS}, we have that $C_{G}(F)\leq F$, and hence $G/F$ is isomorphic to a section of the automorphism group of $F$.

	Since $F_\alpha \neq 1$ by our hypothesis, we conclude that $F$ acts transitively and non-regularly on $\alpha^F$. The action is also faithful, by the hypothesis of this lemma, and we have fixity  at most $4$ by our global hypothesis. Since $F$ is nilpotent, it is not a Frobenius group, and hence it cannot act with fixity $1$.
	
	Assume for a contradiction that $F$ acts with fixity $2$ on $\alpha^F$.
	Then by Lemma~\ref{nilpotent4}, $F$ is a dihedral or semi-dihedral $2$-group of order $|\alpha^F| \cdot |F_\alpha| \geq 4 \cdot 2 = 8$. In particular, the
	auto\-mor\-phism group of $F$ is also a $2$-group (see Theorem 34.8
	in~\cite{Berkovich}). However, then $G$ is
	a $2$-group and hence $G=F$. This contradicts the fact that $G$ acts with fixity~$4$, but $F$
	acts with fixity $2$ on $\alpha^F=\alpha^G=\Omega$.
	
	Next, assume for a contradiction that $F$ acts with fixity $3$ on $\alpha^F$. Then by
	Lemma~\ref{nilpotent4}, $F$ is a $3$-group of maximal class and by Lemma \ref{normaliser}~(iv), $Z := Z(F)$ has order $3$.
	Let $C := C_G(Z)$. Then $|G:C| = |N_G(Z):C_G(Z)| \leq 2$ and Lemma \ref{normaliser} (iv) yields that $C$ acts with fixity $3$ on $\alpha^C$. Let $H \leq G_\alpha$ be a four point stabiliser. Then $C \cap H = 1$, and we conclude that $|G:C| = |H| = 2$. Thus $G = HC$ and  $\Omega = \alpha^G = \alpha^{H C} = \alpha^C$. Hence $C$ acts transitively on $\Omega$. 
    Since $|C:C_\alpha| = |\Omega| \equiv 4$ modulo $|H| = 2$, it follows that $|C|$ and $|\Omega|$ are even. Let $S \in \Syl_2(G)$ be such that $H \leq S$. Then $1 \neq S \cap C \trianglelefteq S$ and there exists an involution $t \in Z(S) \cap C$. Let $V := H \times \<t\> \cong C_2 \times C_2$. Since $[Z,t] = 1$, $Z$ acts semi-regularly on $\FO(t)$ and $|\FO(t)|$ is divisible by $3$. At the same time, $|\Omega|$ and therefore $|\FO(t)|$ is even. Together with the fixity $4$ hypothesis, we obtain that $t$ acts fixed point freely on $\Omega$. Since $|\alpha^F|$ is odd, this is only possible if $\alpha^{tF} \neq \alpha^F$.
	Consider the action of $F:V$ on $\Delta := \alpha^{FV} = \alpha^{H\<t\>F} = \alpha^{\<t\>F}$. Then $|\Delta| = 2 \cdot |\alpha^F|$ has twice odd size and $t$ induces an odd permutation on it. Since $|V| = 4 > 2$, there must exist an involution $s \in V$ that induces an even permutation. The fact that $|\Delta| \equiv 2 \mod 4$, together with the fixity $4$ hypothesis, implies that $|\fix_{\Delta}(s)| = 2$. This is only possible if $|\fix_{\alpha^F}(s)| = |\fix_{\alpha^{tF}}(s)| = 1$. Every element $y \in C_F(s)$ fixes $\fix_{\alpha^F}(s)$ and $\fix_{\alpha^{tF}}(s)$ point-wise, and therefore it fixes at least three points in both $F$-orbits in total. This is only possible if $y = 1$, and we deduce that $C_F(s) = 1$. According to Lemma 8.1.10 in \cite{KS}, $s$ inverts all elements in $F$, and thus $F$ is abelian. But $|Z(F)| = 3 < |F|$, which contradicts our assumption that $F$ acts with fixity $3$.
	Therefore, $F$ acts with fixity $4$ on $\alpha^F$, as stated.  By Lemma~\ref{nilpotent4}, $F$ is a $2$-group of sectional $2$-rank at most $4$.
\end{proof}

We turn our attention to the point stabilisers. Under the hypothesis that the Fitting subgroup acts with fixity $4$, we will see that the point stabilisers are small $\{2,3\}$-groups. For this, we need a result about the existence of certain characteristic subgroups.

\begin{lemma}\label{2groupfix4}
	Suppose that $G$ is a $2$-group that acts transitively, faithfully and with fixity $2$ or $4$ on $\Omega$. Then there exists a characteristic subgroup $N$ of $G$ such that $Z(G) \leq N$ and such that $N$ acts semi-regularly on $\Omega$ with either exactly two or exactly four orbits.
\end{lemma}
\begin{proof}
	Assume for a contradiction that $G$ is a counterexample of minimal order and let $Z := Z(G)$. By Lemma \ref{normaliser}\,(iv) and (vi), $Z$ has order $2$ or $4$ and it acts semi-regularly on $\Omega$. Let $\bar \Omega$ denote the set of $Z$-orbits on $\Omega$.

	Assume that $|\bar \Omega| \leq 4$. Then $|\bar \Omega| > 1$ because otherwise $|\Omega| = |\alpha^Z| = |Z| \in \{2,4\}$, in which case $G$ cannot act with fixity 2 or 4.   Since $|\Omega|$ is a power of $2$ and all $Z$-orbits have the same length, it follows that $|\bar \Omega| \in \{ 2, 4\}$. In particular, $N:=Z$ has all the required properties  and $G$ is not a counterexample.

	 Therefore $|\bar \Omega| > 4$. Then $|\Omega| = |Z| \cdot |\bar \Omega| > |Z| \cdot 4$ and we can apply Lemma \ref{centrereduction}. It states that $\bar G := G / Z$ acts transitively, non-regularly, faithfully and with fixity at most $4$ on $\bar \Omega$. Since $\bar G$ is a $2$-group, this is only possible if $\bar G$ acts with fixity $2$ or $4$, by Lemma \ref{normaliser}\,(iv). 
	Since $G$ is a minimal counterexample, we can apply the lemma to the action of $\bar G$ on $\bar \Omega$ and we find a characteristic subgroup $\bar N \leq \bar G$ that acts semi-regularly on $\bar \Omega$ with exactly two or four orbits. Then the full pre-image $N \leq G$ of $\bar N$ satisfies the conditions listed in the statement and therefore $G$ is not a counterexample.
\end{proof}

\begin{lemma}\label{2^3.3^2}
	Suppose that Hypothesis \ref{hyp} holds and that $F(G)$ acts faithfully with fixity $4$ on $\alpha^{F(G)}$. Then $|G_\alpha|$ divides $2^3 \cdot 3^2$.
\end{lemma}
\begin{proof}
	Let $G$ be a counterexample of minimal order. According to Lemma~\ref{nilpotent4}, $F := F(G)$ is a $2$-group. Since $F$ acts faithfully on $\alpha^F$ and some non-trivial element in $F$ has exactly four fixed points, we see that $|\alpha^F| > 4$. Moreover, the action of $F G_\alpha$ on $\alpha^F = \alpha^{FG_\alpha}$ is also a faithful fixity $4$ action. By the minimality of our counterexample $G$ , it follows that $G = F G_\alpha$ and $F$ acts transitively on $\Omega$.
	Let $S \in \syl_2(G)$ be such that $S$ contains a Sylow $2$-subgroup of $G_\alpha$. Then $F \leq S$ and thus $S$ acts transitively and faithfully with fixity $4$ on $\Omega = \alpha^F$. Now Lemma \ref{nilpotent4} gives $|S_\alpha| \leq 8$ and thus $2^3$ is the largest power of $2$ which could possibly divide $|G_\alpha|$. It remains to show that $G_\alpha$ is a $\{2,3\}$-group of order not divisible by $3^3$.

	Using Lemma \ref{2groupfix4}, we find a normal subgroup $N$ of $F$ such that $Z(F) \leq N$ and such that $N$ acts semi-regularly, with exactly two or four orbits. Note that $N \trianglelefteq G$ because $F$ is characteristic in $G$. Let $n \in \{2,4\}$ and let $\alpha = \delta_1,..,\delta_n \in \Omega$ be such that $\Delta_1 := \delta_1^N$, ..., $\Delta_n := \delta_n^N$ are the distinct $N$-orbits and let $K$ denote the kernel of the action of $G$ on the set of $N$-orbits. Then $G/K$ is isomorphic to a subgroup of $\Sym_4$ and $G_\alpha/K_\alpha$ is isomorphic to a subgroup of $\Sym_3$. As $|\Delta_i| = |N|$ is a power of $2$ for all $i \in \{1,...,n\}$, every $x \in K$ of odd prime order fixes at least one point in every $N$-orbit. Moreover, since $|\fix_{\Delta_1}(x)| \equiv ... \equiv |\fix_{\Delta_n}(x)|$ modulo $o(x)$ and $o(x) \geq 3$ is odd, the fixity $4$ hypothesis implies that $x$ has the same number of fixed points in every $N$-orbit. In particular, $|\fix_{\Delta_1}(x)| = |\FO(x)|/n \leq 4/n \leq 2$.

	Assume for a contradiction that $n = 4$. By our previous considerations, every non-trivial $x \in K_\alpha$ of odd order must have exactly one fixed point on every $N$-orbit. With Lemma \ref{normaliser}\,(ii), it follows that $C_N(x) = 1$ and in particular, $\<x\>$ acts semi-regularly on $Z(F(G))^\#$ by conjugation. Since $Z(F(G)) \leq N$ is a normal subgroup of $G$ of order dividing $4$ according to Lemma \ref{normaliser}\,(iv), this is only possibly if $K_\alpha$ is a $\{ 2,3 \}$-group of order not divisible by $3^2$. Since $|G_\alpha| = |K_\alpha| \cdot |G_\alpha:K_\alpha|$ and $G_\alpha / K_\alpha$ is isomorphic to a subgroup of $\Sym_3$, we see that $G_\alpha$ is a $\{ 2,3 \}$-group whose order is not divisible by $3^3$. In particular, $G$ is not a counterexample and thus $n = 2$. Now $|G:K| = 2$ and $K_\alpha = G_\alpha$.

	Assume for a contradiction that $|\Delta_1| \leq 4$. Then $|\Delta_1| = 4$ because otherwise $|\Omega| \leq 4$,  contrary to the faithful action of $G$. every non-trivial $x \in K_\alpha$ of odd order has at most two fixed points on $\Delta_1$ and must thus act non-trivially. This is only possible if $K_\alpha$ is a $\{ 2, 3 \}$-group of order not divisible by $3^2$. Then, since $K_\alpha = G_\alpha$, we see that $G$ is not a counterexample. Hence $|\Delta_1| \geq 8$.

	Consider the action of $K$ on $\Delta_1$. The fixity $4$ hypothesis, together with the fact that $|\Delta_1| \geq 8$, implies that this action is faithful. Since $N \trianglelefteq K$ is a $2$-group, we have that $N \leq F(K)$ and thus $F(K)$ acts transitively and faithfully on $\Delta_1 = \alpha^F$. Since $1 \neq F_\alpha \leq F(K) \cap K_\alpha$, we can apply Lemma \ref{nilpotent4}: ~$F(K)$ is a $2$-group that acts with fixity $2$ or $4$ on $\Delta_1$. In the first case, $F(K)$ is dihedral or semi-dihedral of order $2 \cdot |\Delta_1| \geq 16$ and therefore $\aut(F(K))$ is a $2$-group (see Theorem 34.8 in \cite{Berkovich}).
    Then it follows that $K$, and therefore $K_\alpha = G_\alpha$, are $2$-groups and $G$ is not a counterexample. 
Consequently, the second case must hold, i.e. $F(K)$ acts faithfully and with fixity $4$ on $\Delta_1$. Now the minimality of our counterexample implies that 
$|K_\alpha|$ divides $2^3 \cdot 3^2$. But $G_\alpha = K_\alpha$, which gives a contradiction. 
\end{proof}

The next example shows that the bound of $2^3 \cdot 3^2$ from the previous lemma does really occur.

\begin{ex}\label{AGammaL24}
    Let $F$ be a field of order $4$ and let $\varphi \in \Gal(F)$ be the non-trivial field automorphism of $F$ of order $2$. Let $H_0 \leq \GL_2(4)$ be the subgroup of lower triangular matrices and $H := H_0:\< \varphi \> \cong ( \Alt_4 \times C_3 ) : C_2$. Consider the natural action of $H$ on $V := F^2$ (where $\varphi$ acts component-wise).
    For all $h \in H_0^\#$, $C_V(h)$ is trivial or it is a 1-dimensional $F$-subspace of $V$ and thus $|C_V(h)| \in \{ 1, 4 \}$. Since every involution $t \in H_0$ is contained in $O_2(H_0) \in \syl_2(H_0)$ and can be written in the form $t = \begin{pmatrix} 1 \\ * & 1 \end{pmatrix}$, we see that $|C_V(t)| = 4$. By Proposition 4.9.1 (d) in \cite{GLS3}, every involution $t \in H \setminus H_0$ is $\GL_2(4)$-conjugate to $\varphi$ and thus satisfies $|C_V(t)| = 4$. With Lemma \ref{RegNSG}, it now follows that $G := V:H$ acts with fixity $4$ on $\Omega := G / H$ by right multiplication. Moreover, the Fitting subgroup $F(G) = O_2(G) = V:O_2(H_0)$ acts transitively and with fixity $4$ on $\Omega$ and the point stabilisers have order $|H| = 2^3 \cdot 3^2$.
\end{ex}

\medskip

Next, we look at the case where the Fitting subgroup does not act faithfully on its orbits. We start by showing that this is only possible if the group acts on a set of size at most $28$. We then use \texttt{GAP}, more specifically the Transitive Groups Library~\cite{TransGrp}, and in this way we can find all examples and determine the structure of the point stabilisers.

\begin{lemma}\label{ZenFG1}
    Suppose that Hypothesis \ref{hyp} holds and that $F(G)$ does not act faithfully on $\alpha^{F(G)}$. Then $|\Omega| \leq 28$.
\end{lemma}
\begin{proof}
    Assume for a contradiction that $|\Omega| > 28$ and consider the non-trivial abelian subgroup $Z := Z(F(G))$. By Lemma \ref{Ffaithful}, $Z_\alpha \neq 1$, and Lemma~\ref{normalaction} gives for all $\omega \in \Omega$ that $(Z, \alpha^Z)$ and $(Z, \omega^Z)$ are isomorphic as permutation groups. In particular $|Z_\omega| = |Z_\alpha|$ and $|\omega^Z| = |\alpha^Z|$. We also recall that $Z$ is abelian, hence $Z_\alpha$ fixes $\alpha^Z$ element-wise. Then  Hypothesis \ref{hyp} implies that
     $|\alpha^Z| \leq 4$. We count the set 
    \[
        \Gamma := \{ (\omega,z) \in \Omega \times Z^\# \mid \omega^z = \omega \} = \dot\bigcup_{z \in Z^\#} \FO(z) \times \{z\} = \dot\bigcup_{\omega \in \Omega} \{ \omega \} \times Z_\omega^\#
    \]
    in two different ways. Since each $z \in Z^\#$ satisfies $|\FO(z)| \leq 4$ by hypothesis, we obtain that
    \[
        4 \cdot |Z^\#| \geq \sum_{z \in Z^\#} |\FO(z)| = |\Gamma| = \sum_{\omega \in \Omega} |Z_\omega^\#| = |\Omega| \cdot |Z_\alpha^\#| > 28 \cdot |Z_\alpha^\#|. 
    \]
    Rearranging gives the contradiction
    \begin{align*}
        7 &< \frac{|Z^\#|}{|Z_\alpha^\#|}
		= \frac{|Z| - |Z:Z_\alpha|}{|Z_\alpha| - 1} + \frac{|Z:Z_\alpha| - 1}{|Z_\alpha| - 1}
		= \frac{|Z:Z_\alpha| \cdot (|Z_\alpha| - 1)}{|Z_\alpha| - 1} + \frac{|Z:Z_\alpha| - 1}{|Z_\alpha| - 1} \\
		& \leq |Z:Z_\alpha| + (|Z:Z_\alpha| - 1)
		= 2 \cdot |\alpha^Z| - 1 \leq 7.
    \end{align*}
\end{proof}

In Examples \ref{ExRow8and20} and \ref{ExRow28}, we already saw that the bound $|\Omega| \leq 28$ in the previous lemma is sharp.

\bigskip
\begin{lemma}\label{FittingGAP}
	Suppose that Hypothesis \ref{hyp} holds, that $F(G)\cap G_{\alpha}\neq 1$ and that $F(G)$ does not act faithfully on $\alpha^{F(G)}$.
	Then $F(G)$ has one of the following isomorphism types:
    \begin{enumerate}
    \item
    Elementary abelian of order $4$, $8$, $9$ or $16$,
	
    \item $C_4\times C_4$, or 
    \item $(C_4 \times C_4):C_2$.
    \end{enumerate}
	Moreover, Table~\ref{TableNotFaithful} lists all groups $G$ that fulfil the hypothesis of this lemma, together with all possible point stabiliser structures in each case.
\end{lemma}
\begin{proof}
	By Lemma~\ref{ZenFG1}, $|\Omega| \leq 28$. By Lemma \ref{normalaction}, the action of $F(G)$ on $\alpha^{F(G)}$ is in one-to-one correspondence with the action of $F(G)$ on $\omega^{F(G)}$ for every $\omega\in \Omega$. Therefore, the calculation is independent of $\alpha$ and we can use the point~$1$ for our calculations without loss of generality. The following  \texttt{GAP} code uses the  \texttt{GAP} program at the end of Section 2 together with the
	Transitive Groups Library~\cite{TransGrp} and returns a list \texttt{Gs}. This list contains information about every group $G$ that acts transitively and with fixity $4$
	on a set of size at most~$28$ while satisfying $Z(F(G)) \cap G_\alpha \neq 1$. By Lemma \ref{Ffaithful}, the last condition is equivalent to $F(G)$ not acting faithfully on $\alpha^{F(G)}$. Thus \texttt{Gs} is the full list of examples.
	Afterwards, we map every group in \texttt{Gs} to a tuple of three IDs in the Small Groups Library~\cite{SmallGrp} and eliminate duplicates.
	The first entry represents the group itself, the second entry represents its Fitting subgroup, and the third entry represents a point stabiliser.
\begin{verbatim}
Gs := AllTransitiveGroups(
    NrMovedPoints, [2..28],
    HasFixity(4),  true,
    G -> IsTrivial(Stabiliser(Centre(FittingSubgroup(G)), 1)), false
);;
Set(List(Gs, G -> [ IdGroup(G),
                    IdGroup(FittingSubgroup(G)),
                    IdGroup(Stabiliser(G,1)) ] ));
\end{verbatim}
	Since the IDs are unique up to isomorphism, they can be used as a basis
	to determine the isomorphism types of the groups they represent.
	For Table~\ref{TableNotFaithful} the results %
	are sorted in such a way that isomorphic groups just appear once,
	but with all their
	actions that fulfil the hypothesis of this lemma.
	For example, the structure of the Fitting subgroup with ID $[32,34]$ in
	the Small Groups Library~\cite{SmallGrp} is
	denoted by $(C_4\times C_4):C_2$.
\end{proof}

This finishes off our analysis of the case where some non-trivial element of the Fitting subgroup fixes a point.

\subsection{$F(G)$ acting semi-regularly.}

This case is much easier to handle. 

\begin{lemma}\label{FGsemireg}
Suppose that Hypothesis \ref{hyp} holds and that $F(G)$ acts semi-regularly.
		Then every 
		$x \in G_\alpha^\#$ that fixes a point satisfies $|C_{F(G)}(x)|\leq 4$, and one of the following holds:
		\begin{enumerate}[\quad (1)] %
			\item For all $p\in \pi(G_{\alpha})$ the Sylow $p$-subgroups of $G_{\alpha}$ have $p$-rank $1$.
			\item $F(G)=O_2(G) \times O_3(G)$.
	\end{enumerate}
	\end{lemma}
	
	\begin{proof}
	Let $x\in G_\alpha^\#$. Since $F(G) \cap G_{\alpha}=1$, we can use Lemma~\ref{normaliser}~(ii) to see that $|C_{F(G)}(x)|=|C_{F(G)}(x):C_{F(G) \cap G_{\alpha}}(x)|\leq 4$.
	
	Suppose that (1) does not hold. Then there exists $p\in \pi(G_{\alpha})$ such that $G_{\alpha}$ contains an
	elementary abelian subgroup $X$ of order $p^2$. Assume that there exists a prime $r\in \pi(F(G))$ such that
	$r\geq 5$. If $r=p$, then Lemma~\ref{normaliser}~(v) yields that $G_{\alpha}$ contains a Sylow $r$-subgroup of $G$, and
	then it follows that $1 \neq O_r(G)\le G_\alpha$, contradicting the transitive and faithful action of $G$.
	Therefore, $X$ acts coprimely on $O_r(G)$ and it follows that $O_r(G)=\langle C_{O_r(G)}(x)
	\mid x\in  X\setminus \{1\}\rangle$. Since for all non-trivial $x\in  X$, $ 4\geq |C_{F(G)}(x)|\geq
	|C_{O_r(G)}(x)|$ and $r\geq 5$, we deduce that $C_{O_r(G)}(x)=1$. Thus, $O_r(G)=1$, contrary to the
	fact that $r\in \pi(F(G))$.
\end{proof}

\section{The structure and action of the components}

We begin with some general results, and then we look at a number of cases in separate sub-sections, depending on how exactly the components of $G$ act on $\Omega$. 

\subsection{Quasi-simple groups acting with low fixity}

\bigskip
\begin{lemma}\label{quasisimpleFix23}
	Let $E$ be a quasi-simple group that acts transitively and with fixity $2$ or $3$ on a set $\Omega$ of size at least $4$. Then $E$ is simple.
\end{lemma}
\begin{proof}
	
	If $E$ acts with fixity $2$ on $\Omega$, then, by Theorem~5.1 in~\cite{MW}, the quasi-simple group~$E$ is in fact simple.
	Therefore, we suppose from now on that $E$ acts with fixity $3$ on~$\Omega$.

	Assume for a contradiction that $E$ is not simple. Then $Z := Z(E) \neq 1$ and by Lemma~\ref{normaliser}~(iv), the number $|Z|$ divides $3$ and therefore equals $3$. Let $\bar{E}:=E/Z$ and let $\bar{\Omega}$ denote the set of $Z$-orbits. Lemma \ref{centrereduction} states that $\bar{E}$ acts transitively and faithfully with fixity $1$, $2$ or $3$ on $\bar{E}$, and we note that fixity $1$ is impossible by Remark \ref{CompInd}~(b). Let $\alpha \in \Omega$ and $\bar{\alpha} := \alpha^Z \in \bar{\Omega}$. 

	Assume for a contradiction that $\bar{E}$ acts with fixity $2$ on $\bar{\Omega}$. By Lemma \ref{centrereduction}, the two-point stabilisers in $\bar{E}$ are cyclic of order $3$. If $|\bar{E}_{\bar{\alpha}}| \neq 3$, then Lemma 2.15 in \cite{MW} implies that $\bar{E}_{\bar{\alpha}}$ is a Frobenius group with complements of order $3$. Going through the list of fixity $2$ actions in Table \ref{TableAllFix} shows that either $\bar{E} \cong \Alt_5$ with $\bar{E}_{\bar{\alpha}}$ isomorphic to $C_3$ or $\Alt_4$, or $\bar{E} \cong \PSL_2(7)$ with $\bar{E}_{\bar{\alpha}}$ isomorphic to $C_3$ or $C_7 : C_3$. But $\Alt_5$ and $\PSL_2(7)$ do not possess triple covers. Hence $\bar{E}$ cannot act with fixity $2$ on $\bar{\Omega}$.

	We deduce that $\bar{E}$ acts with fixity $3$ on $\bar{\Omega}$. Then Lemma \ref{centrereduction} implies that the three-point stabilisers are cyclic of order $3$ and that every non-trivial element in $\bar{E}_{\bar \alpha}$ has either just one or precisely three fixed points on $\bar{\Omega}$. Since there exists some element of order $3$ with exactly three fixed points, it follows that $|\bar{\Omega}|$ must be divisible by $3$, i.e. $\bar{E}_{\bar{\alpha}}$ does not contain a Sylow $3$-subgroup of $\bar{E}$. With Table \ref{TableAllFix}, it follows that either $\bar{E} \cong \Alt_6$ with $\bar{E}_{\bar{\alpha}}$ isomorphic to $\Sym_4$ or $\Alt_5$ and $|\bar{\Omega}| = 15$ or $6$, or $\bar{E} \cong \Alt_7$ with $\bar{E}_{\bar{\alpha}} \cong \PSL_2(7)$ and $|\Omega| = 15$. In all cases, $|\bar{\Omega} \setminus \{\bar{\alpha}\}| = |\bar{\Omega}| - 1$ is not divisible by $4$, but $|\bar{E}_{\bar{\alpha}}|$ is divisible by $4$. Hence there exists an involution in $\bar{E}_{\bar{\alpha}}$ with at least two fixed points, in contradiction to Lemma \ref{centrereduction}. Thus $\bar{E}$ cannot act with fixity $3$ on $\bar{\Omega}$, which gives a contradiction, and hence $E$ must be simple.

\end{proof}

For fixity $4$, the situation is different and allows for more variety, as the next few results show.

\begin{lemma}\label{quasisimple}
	Let $E$ be a quasi-simple group that acts transitively and with fixity $4$ on a set $\Omega$ of size at least $5$. Then $E$ is simple or isomorphic to $2^{\cdot}\Sz(8)$, $2^{\cdot}\PSL_3(4)$ or $SL_2(q)$ for some odd prime power $q$. Moreover, if $Z(E) \neq 1$, then $E/Z(E)$ acts with fixity $2$ on the set of $Z(E)$-orbits, and the point stabilisers have odd order.
\end{lemma}
\begin{proof}
	
	Suppose that $E$ is not simple. Then $Z := Z(E)$ has order $2$ or $4$ according to Lemma \ref{normaliser} (iv). Let $\bar{E} := E / Z$ and let $\bar{\Omega}$ denote the set of $Z$-orbits. Furthermore, let $\alpha \in \Omega$ and $\bar{\alpha} := \alpha^{Z} \in \bar{\Omega}$. By Lemma \ref{centrereduction}, $\bar{E}$ acts transitively and faithfully with fixity $1$, $2$, $3$ or $4$ on $\bar{\Omega}$ and we note that fixity $1$ is impossible by Remark \ref{CompInd}~(b).

	Assume for a contradiction that $\bar{E}$ acts with fixity $3$ or $4$ on $\bar{\Omega}$.
	Let $H$ be the full preimage of a three point stabiliser or a four point stabiliser $\bar{H} \leq \bar{E}_{\bar{\alpha}}$, respectively. Lemma \ref{centrereduction} (i) and (iii)  yields that $\bar{H}$ is a $2$-group and that $H = Z \times H_\alpha$. In particular, $\bar{E}_{\bar{\alpha}}$ has even order.\\
Assume for a contradiction that $\bar{E}$ is isomorphic to $\PSL_2(q)$ for some prime power $q \geq 5$. Since $|Z| \in \{2,4\}$, this is only possible if $E \cong \SL_2(q)$, where $q$ is odd. Now $E$ contains a unique involution, in contradiction to the fact that $|\Omega_1(H)| = |\Omega_1(Z) \times \Omega_1(H_\alpha)| \geq 4$. Hence $\bar{E}$ is not isomorphic to $\PSL_2(q)$.\\
	We now go through the list in Table \ref{TableAllFix} and exclude all groups isomorphic to $\PSL_2(q)$, for some $q \geq 5$, and all cases where the point stabilisers have odd order. We also exclude $\M_{11}$ and $\PSU_3(3)$ because these groups do not possess double covers. This leaves us with two possibilities: $\bar{E} \cong \Alt_7$, where $\bar{E}_{\bar{\alpha}}$ is isomorphic to $\PSL_2(7)$ or to $\Alt_6$ and $\bar{E} \cong \M_{12}$, where $\bar{E}_{\bar{\alpha}} \cong \M_{11}$. In each case, $\bar{E}$ contains an element of order $3$ with three or four fixed points on $\bar{\Omega}$, which is impossible because $\bar H$ is a $2$-group. 

	We are left with the case where $\bar{E}$ acts with fixity $2$ on $\bar{\Omega}$. If $\bar{E} \cong \Sz(q)$ for some $q$, then $\bar{E} \cong \Sz(8)$, because the other groups do not possess double covers. According to Table \ref{TableAllFix}, $\bar{E}_{\bar{\alpha}}$ is cyclic of order $7$. If $|Z| = 4$, then we obtain a contradiction to Lemma \ref{centrereduction} (iii). Hence $|Z| = 2$ and $E \cong 2.\Sz(8)$. Similarly, if $\bar{E} \cong \PSL_3(4)$, then $\bar{E}_{\bar{\alpha}}$ is cyclic of order $5$ and Lemma \ref{centrereduction} gives that $E = 2.\PSL_3(4)$. According to Table \ref{TableAllFix}, $\bar{E}$ is isomorphic to $\PSL_2(q)$ in every other case. Since $|Z| \in \{2,4\}$, this is only possible if $E \cong \SL_2(q)$. If the point stabilisers had even order, then the full preimage $Z \times E_\alpha$ of $\bar{E}_{\bar{\alpha}}$ would contain an elementary abelian subgroup of order $4$, which is impossible by the structure of $\SL_2(q)$. This concludes the proof.
\end{proof}

\bigskip
\begin{ex}\label{2.PSL_2.Sz}
	Let $E := \PSL_3(4)$ or $E := \Sz(8)$. Using Table \ref{TableAllFix}, we see that both groups admit a fixity $2$ action with cyclic point stabilisers of order $5$ or $7$, respectively. Let $G$ be a quasi-simple group with centre of order $2$ such that $\bar G := G/Z(G) \cong E$, and let $H \leq G$ be cyclic of order $5$ or $7$ such that $\bar{G}$ acts with fixity $2$ on $\bar{G}/\bar{H}$. The full pre-image of $\bar{H}$ is given by $H \times Z(G)$. Then, since $(|H|,|Z(G)|) = 1$, it follows that $H$ is strongly closed in $H \times Z(G)$ with respect to $G$. Now Lemma \ref{centreextension} gives that $G$ acts with fixity $4$ on the coset space $G/H$ via right multiplication.
\end{ex}

\bigskip
\begin{lemma}\label{SLquasisimple}
	Let $q \geq 5$ be an odd prime power and let $G=\SLi_2(q)$.
	Suppose that $G$ acts transitively and faithfully on a set $\Omega$.
	Then $G$ acts with fixity $4$ on $\Omega$ if and only if one of the following cases occurs:
	\begin{enumerate}[\quad (1)]
		\item $q=5$ and the point stabilisers are cyclic of order $3$ or $5$.
		\item $q \equiv 1 \mod 4$, $q\geq 9$, and the point stabilisers are cyclic of order $\frac{q+1}{2}$.
		\item $q \equiv -1 \mod 4$, $q\geq 7$, and the point stabilisers are cyclic of order $\frac{q-1}{2}$, or they are 
		a semi-direct product of an elementary abelian group of order $q$ with a cyclic group of order $\frac{q-1}{2}$.
	\end{enumerate}
\end{lemma}
\begin{proof}
	Let $Z := Z(G)$, $\bar{G} := \PSL_2(q) = G / Z$ and $\bar{\Omega} := \{ \omega^{Z} \mid \omega \in \Omega \}$. Furthermore, let $\alpha \in \Omega$ and $\bar{\alpha} := \alpha^{Z}$.
	Suppose that $G$ acts with fixity $4$ on $\Omega$. Since $G$ is quasi-simple, Lemma \ref{centrereduction} is applicable, and it gives that $\bar{G}_{\bar{\alpha}} \cong G_{\alpha}$. 
	Moreover, by Lemma \ref{quasisimple}, $G_\alpha$ has odd order and we can go through the different cases described in Table \ref{TableAllFix}: In each case $G_\alpha$ is isomorphic to one of the groups listed in the statement.
    
	Conversely, if $G_\alpha$ is one of the groups described in the statement, then $G_\alpha \cap Z = 1$ and $G_\alpha = O^2(Z G_\alpha) = O(Z G_\alpha)$ is strongly closed in $Z G_\alpha$. Then Table \ref{TableAllFix} tells us that $\bar{G}$ acts with fixity $2$ on $\bar{G}/\bar{G_\alpha}$ by right multiplication. Finally, Lemma \ref{centreextension} implies that $G$ acts with fixity $|Z| \cdot 2 = 4$ on $\Omega$.
\end{proof}

Strictly speaking, our previous arguments use the fact that all examples from Table \ref{TableAllFix} actually occur with the given fixity. This can be checked by looking at the individual groups and their subgroup structure.

\bigskip\noindent
Lemmas \ref{SLquasisimple} and \ref{quasisimple}, together with Example~\ref{2.PSL_2.Sz}, give a full classification of all non-simple, quasi-simple groups that allow for an action with fixity $4$. Together with Table \ref{TableAllFix} this gives us a full classification of all quasi-simple groups that act transitively and with fixity $4$.

\subsection{Semi-regular components and the structure of $E(G)$}

In the previous two sections we have seen that it is often useful to argue with Lemma~\ref{normaliser}~(i) or (ii), applied to the normaliser or the centraliser of a
subgroup of a point stabiliser.
Thus, it is unsurprising that the size of a centraliser will play a key role in the analysis of the structure of possible components. This is why we begin this subsection with a lemma that is interesting in its own right. The idea for the proof is due to Gernot Stroth who suggested to use a version of the Brauer-Fowler-Theorem together with a library of primitive groups.

\bigskip
\begin{lemma}\label{ceninvsimple}
	Let $E$ be a non-abelian simple group. Let $x$ be an  automorphism
	of~$E$ of order $2$. If $E$ is not isomorphic to $\Alt_5$, then  $|C_E(x)| > 4$.
\end{lemma}

\begin{proof}
	Assume for a contradiction that $E$ is not isomorphic to $\Alt_5$ and that the order of $C_E(x)$ is at most $4$.
	Since $E$ is non-abelian simple, it is possible to identify $E$ with the subgroup $\Inn(E) \leq \aut(E)$ of inner automorphisms. In particular, we have $C_E(x) \cong C_{\Inn(E)}(x)$ and our assumption can be restated as $|C_{\Inn(E)}(x)| \leq 4$. Using this idea, the following \texttt{GAP} code tests whether or not the non-abelian simple group $E$ possesses an automorphism $x$ of order $2$ such that $|C_E(x)| \leq 4$. Note that it is sufficient to check one representative of every conjugacy class in $\aut(E)$.
\begin{verbatim}
HasSuitableAutomorphism := E -> not IsEmpty(Filtered(
    List(ConjugacyClasses(AutomorphismGroup(E)), Representative),
    x -> Order(x) = 2 and
         Order(Centraliser(InnerAutomorphismGroup(E), x)) <= 4
));;
\end{verbatim}
	 We start by showing that $E$ is not isomorphic to $\Alt_n$ for some $n \geq 6$. For $n \in \{6, 7\}$, we use \texttt{GAP} and see that \texttt{HasSuitableAutomorphism(AlternatingGroup(n))} evaluates to \texttt{false} for \texttt{n := 6} and \texttt{n := 7}. For $n \geq 8$, Theorem 5.2.1 in \cite{GLS3} states $\aut(E) \cong \Sym_n$ and we identify $E$ with $\Alt_n$ and $x$ with the corresponding element in $\Sym_n$. Since $x$ is an involution, it can be written as a product of disjoint transpositions, and we may assume without loss that $x = (12)(34)...(2\ell-1,2\ell) \in \Sym_n$ for some $\ell \in \N$. But then $\<(12),(34),(56),(78)\> \cap \Alt_n \leq C_E(x)$ is elementary abelian of order $8$, contrary to our assumption that $|C_E(x)| \leq 4$. Thus $E$ is not isomorphic to an alternating group. 

	If $x$ induces an inner automorphism on $E$, then we define $H := E$ and identify $x$ with the corresponding element in $E$. Otherwise, we define $H := E:\<x\>$. In both cases, $H$ is almost simple, $E = H'$ has index at most $2$ in $H$, and by assumption $|C_E(x)| \leq 4$. This implies that $|C_H(x)| \leq 8$.
	By Theorem 1.5 in Chapter 5 of \cite{Suz} (using the method of Brauer and Fowler), there exists a non-trivial $w \in H$ such that $|H : C_H(w)| \leq |C_H(x)|^2 \leq 8^2 = 64$. 

	If $w \in E$, then $C_E(w) \neq E$ because $E$ is non-abelian simple, and we find a maximal subgroup $M$ of $E$ such that $C_E(w) \leq M < E$. Note that $|E : C_E(w)| \leq |H : C_H(w)| \leq 64$ by our previous considerations and that $M$ is core-free because $E$ is simple. If $w \notin E$, then we find a maximal subgroup $M$ of $H$ such that $C_H(w) \leq M < H$. Since $w \in M \setminus E$ and $E$ is the only non-trivial proper normal subgroup of $H$, the subgroup $M$ must also be core-free. In any case, we find an almost simple group $G$ such that $E = G'$ and an element $w \in G$ as well as a maximal core-free subgroup $M$ of $G$ such that $C_G(w) \leq M < G$ and $|G:C_G(w)| \leq 64$. In particular, the action of $G$ on $G/M$ by right multiplication is transitive, faithful and primitive, and $|G/M| \leq 64$. Since $1 \neq w \in C_M(w) = C_G(w)$, the point stabiliser $M$ contains a subgroup of index  $|M:C_G(w)| = |G:C_G(w)| / |G:M| \leq 64 / |G/M|$ with non-trivial centre. The following \texttt{GAP} code tests whether the latter condition is satisfied for a given transitive permutation group $G$ with point stabiliser $M$:
\begin{verbatim}
HasLowIndexCentraliser := G -> not IsEmpty(Filtered(
    LowIndexSubgroups(Stabiliser(G,1), Int(64 / NrMovedPoints(G))),
    M -> not IsTrivial(Centre(M))
));;
\end{verbatim}
	To summarise, $E$ must appear as the derived subgroup of a primitive, transitive and almost simple permutation group $G$ of degree at most $64$ which also satisfies the centraliser condition mentioned above. Moreover, we also saw that $E$ is not an alternating group and therefore $G$ not a symmetric group.
	The following \texttt{GAP} code uses the Primitive Permutation Groups Library \cite{PrimGrp} to iterate over all such groups $G$. It then discards the cases where $E = G'$ does not admit a suitable automorphism.
\begin{verbatim}
AllPrimitiveGroups(
    NrMovedPoints, [1..64],             IsAlmostSimple, true,
    IsSymmetricGroup, false,            IsAlternatingGroup, false,
    HasLowIndexCentraliser, true,
    G -> HasSuitableAutomorphism(DerivedSubgroup(G)), true
);
\end{verbatim}

	Evaluating this expression returns an empty list. Thus, no non-abelian simple group besides $\Alt_5$ admits an involutionary automorphism with centraliser of order at most $4$.
\end{proof}

Later, we will distinguish cases along the possibilities for the action of $E(G)$ on its orbits. In the remainder of this subsection, we prepare for this work with a few more general results.

\begin{lemma}\label{semiregcomp}
	Suppose that Hypothesis \ref{hyp} holds and that $E \leq G$ is a quasi-simple subgroup that acts semi-regularly on $\Omega$. Suppose that there exists some non-trivial $x \in N_G(E)$ such that $|\FO(x)| = 4$. Then $E$ is isomorphic to $\Alt_5$ or $\SLi_2(5)$ and $|C_E(x)| = 4$.
\end{lemma}
\begin{proof}
	Since the number of fixed points of non-trivial elements does not exceed four, we may assume without loss that $x$ has prime order. We note that $C_E(x)$ stabilises the set $\FO(x)$, which has size $4$, and then our hypothesis implies that $C_E(x)$ acts semi-regularly on it. In particular $|C_E(x)|$ divides $4$. Now the main result in \cite{Fu1984} states that $x$ is an involution because $E$ is not soluble.
	For $\bar E := E / Z(E)$, Lemma 2.24 in \cite{Isaacs} gives that $|C_{\bar E}(x)| \leq |C_E(x)| \leq 4$. Then Lemma \ref{ceninvsimple} yields that $\bar{E} \cong \Alt_5$. Since $\Alt_5$ has a Schur multiplier of order $2$ and a unique double cover, namely $\SL_2(5)$, we obtain that $E$ is isomorphic to one of these two groups. The structure of their automorphism groups implies that $|C_E(x)| \geq 4$ and therefore $|C_E(x)| = 4$.
\end{proof}

\begin{lemma}\label{normalquasiproduct}
	Suppose that Hypothesis \ref{hyp} holds, that $n \in \N$, $n \geq 2$, and that $L_1,...,L_n$ are pair-wise distinct components of $G$ such that their product $N := L_1 \ast \cdots \ast L_n$ is a normal subgroup of $G$. Then $N$ acts faithfully and non-regularly on $\alpha^N$.
\end{lemma}
\begin{proof}
	First we note that $N \not\leq G_\alpha$ and thus $|\alpha^N| > 1$, because the action of $G$ on $\Omega$ is faithful. In fact, we know that $|\alpha^N| > 4$ by Remark \ref{CompInd}~(a). Then our main hypothesis implies that the action of $N$ on $\alpha^N$ is faithful.\
	Assume for a contradiction that $N$ acts regularly on $\alpha^N$ and let $x \in G_\alpha$ be of prime order $p$ such that $|\FO(x)| = 4$. If there exists some $i \in \{1,...,n\}$ such that $L_i^x \neq L_i$, then $L_i,L_i^x,...,L_i^{x^{p-1}}$ are pairwise distinct components of $G$ and $L := \{ l \cdot l^x \cdots l^{x^{p-1}} \mid l \in L_i \} \leq N$ is a quasi-simple subgroup of $G$ that is centralised by $x$. Now $60 \leq |L| \leq |C_N(x)|$, which contradicts Lemma \ref{normaliser}~(ii) and our assumption that $N$ acts regularly on $\alpha^N$. Consequently, every component in $N$ is normalised by $x$. Now we can apply Lemma \ref{semiregcomp} and obtain $|C_{L_i}(x)| = 4$ and $|L_i \cap L_j| \leq 2$ whenever $i,j \in \{1,...,n\}$ are distinct. In particular, $|C_N(x)| \geq |C_{L_1 \ast L_2}(x)| \geq \frac{4 \cdot 4}{2} = 8$, which is impossible by Lemma \ref{normaliser}~(ii). Hence $N$ does not act regularly on $\alpha^N$, which concludes the proof.
\end{proof}

\begin{prop}\label{onecomp}
	Suppose that the group $G$ acts transitively, faithfully and non-regularly with fixity at most $4$ on a set $\Omega$ and that $E(G) \neq 1$. Then $E(G)$ is quasi-simple.
\end{prop}
\begin{proof}
	Assume otherwise and let $G$ be a minimal counterexample. If $G$ acts with fixity $1$, then $G$ is a Frobenius group and $E(G) = 1$. If $G$ acts with fixity $2$ or $3$, then Lemma 4.4 in \cite{HW} and Lemma 4.3 in \cite{MW3corr}, respectively,
	 state that $E(G)$ is quasi-simple. 
	Hence $G$ acts with fixity $4$.
	Let $n \in \N$ and $L_1,...,L_n$ be components of $G$ such that $E(G) = L_1 \ast \dots \ast L_n$. Then $n \geq 2$ because $G$ is a counterexample. Now we proceed in multiple steps.\\

	\emph{(1) $G = L_1 \ast L_2$.}\\
	According to Lemma \ref{normalquasiproduct}, the normal subgroup $E(G)$ would be a counterexample to our result if it was a proper subgroup. Hence $G = E(G)$ because of the minimality of $G$. If there are three or more components, then this implies that $L_1 \ast L_2$ is a proper normal subgroup of $G = E(G)$, which would give a smaller counterexample by Lemma \ref{normalquasiproduct}.\\

	\emph{(2) $Z(G) = 1$.}\\
	Lemma \ref{centrereduction} states that $G/Z(G)$ acts transitively, faithfully, non-regularly and with fixity at most $4$ on the set of $Z(G)$-orbits. If $Z(G)$ was non-trivial, then $G/Z(G)$ would be a smaller counterexample, contrary to the choice of $G$.\\

	\emph{(3) $L_1$ and $L_2$ act semi-regularly on $\Omega$, and both groups are isomorphic to $\Alt_5$.}\\
	Assume otherwise and choose $i \in \{1,2\}$ such that there exists a non-trivial $x \in L_i \cap G_\alpha$ for some $\alpha \in \Omega$. Since $[x,L_{3-i}] = 1$, Lemma \ref{normaliser}~(ii) implies that a subgroup of index at most $4$ in $L_{3-i}$ is contained in $G_\alpha$. According to Remark \ref{CompInd}~(a), this is only possible if $L_{3-i} \leq G_\alpha$, which contradicts the faithful action of $G$. Now Lemma \ref{semiregcomp} is applicable: both components are isomorphic to $\Alt_5$ or to $\SLi_2(5)$. Since $Z(G) = 1$ by (2), it follows that $L_1 \cong L_2 \cong \Alt_5$.\\

	\emph{(4) $G_\alpha$ is elementary abelian of order $4$.}\\
	If $5 \in \pi(G_\alpha)$, then Lemma \ref{normaliser}~(v) implies that a Sylow $5$-subgroup $P$ of $G$ is contained in $G_\alpha$. But then $1 \neq L_1 \cap S \leq L_1 \cap G_\alpha$, contrary to (3). If $2^3$ or $3^2$ were divisors of $|G_\alpha|$, then $|L_1 \cap G_\alpha|$ would be divisible by $2$, or by $3$, respectively, and this contradicts (3) again. Hence $|G_\alpha|$ divides $12$.\\
	If $|G_\alpha| = 3$, then the number of fixed points of every non-trivial element in $G_\alpha$ would be divisible by $3$, because $|\Omega| = |G:G_\alpha|$ is divisible by $3$.  
    But this is impossible because $G$ acts with fixity $4$. Hence there exists some involution $t \in G_\alpha$. Choose $S \in \syl_2(G)$ to be such that $t \in S$. Then $t \in Z(S)$ and Lemma \ref{normaliser}~(ii) yields that $|S:S_\alpha| \leq 4$. It follows that $|S_\alpha|$, and therefore $|G_\alpha|$ as well, are divisible by $4$. Now $|G_\alpha| \in \{ 2^2, 2^2 \cdot 3 \}$. If $|G_\alpha|= 12$, then we take 
	$P \in \syl_3(G_\alpha)$ and we see using (3) that $P \cap L_1 = P \cap L_2 = 1$. Therefore $N_G(P) \cong (C_3 \times C_3) : C_2$ and Lemma \ref{normaliser}(i) states that $|N_G(P):N_{G_\alpha}(P)| \leq 4$, which is only possible if $|N_{G_\alpha}(P)| = 6$. But then $G_\alpha$ contains exactly two Sylow $3$-subgroups, which contradicts Sylow's theorem. This proves that $|G_\alpha|=4$. The 2-structure of $G$ (see (1) and (3)) implies that $G_\alpha$ is not cyclic.\\

	We can now deduce a final contradiction.
	Since $G \setminus (L_1 \cup L_2)$ has a unique conjugacy class of involutions, we know that every involution $t \in G_\alpha^\#$ satisfies $|C_G(t)| = 16$ (by (3)), hence we can apply Lemma \ref{numberfpt}. Since $G_\alpha$ is elementary abelian, it follows that $t$ has $\frac{3 \cdot 16}{4} = 12 > 4$ fixed points, contrary to the fact that $G$ acts with fixity $4$. Thus there is no counterexample to the assertion of the proposition.
\end{proof}

\bigskip
\begin{lemma}\label{Enonreg}
Suppose that Hypothesis \ref{hyp} holds and that $E(G) \neq 1$. Then $E(G)/Z(E(G))\cong \Alt_5$ or $E(G)\cap G_{\alpha}\neq 1$.
		
\end{lemma}
\begin{proof}
	Suppose that $E(G)\cap G_{\alpha}=1$.
	Since $E(G)\trianglelefteq G$, it follows that $E(G)$ acts semi-regularly on~$\Omega$.
	By Proposition~\ref{onecomp}, $E(G)$ is quasi-simple and therefore the unique component of $G$.
	Then by Lemma~\ref{semiregcomp}, $E(G)/Z(E(G))\cong \Alt_5$.
\end{proof}

\bigskip\noindent
In the case that the unique component acts semi-regularly, we can give some more information regarding the structure of the point stabilisers.

\bigskip
\begin{lemma}\label{regularGaoptions}
	Suppose that Hypothesis \ref{hyp} holds and that $E(G)$ acts semi-regularly on $\Omega$.
	If $E(G)\cong \Alt_5$, then the point stabilisers of $G$ are of isomorphism type
	$C_2, E_4, C_4, \Sym_3$, or $\Alt_4$.
	If $E(G)\cong \SLi_2(5)$, then the point stabilisers of~$G$ are of isomorphism type
	$C_2, E_4$, or $C_4$.
\end{lemma}

\begin{proof}
    Suppose that $E := E(G) \neq 1$. Since $E$ acts semi-regularly, we have that $G_\alpha \cap E = 1$. Moreover, Lemma \ref{normaliser}\,(ii) states that $|C_E(x)| \leq 4$ for all non-trivial $x \in G_\alpha$. Since $|E| > 4$, it follows that the non-trivial elements in $G_\alpha$ induce non-trivial automorphisms of $E$. In this way, we can identify $G_\alpha$ with a subgroup of $\aut(E)$.

    Suppose that $E \cong \Alt_5$ and identify $\aut(E)$ with $\Sym_5$, $E$ with $\Alt_5$ and $G_\alpha$ with a subgroup $H \leq \Sym_5$. By Lemma \ref{normaliser}(ii), we have that $|C_E(x)| \leq 4$ for all $x \in H^\#$. If $x$ is a $5$-cycle, then $|C_E(x)| = 5 > 4$, and if $x$ is a transposition, then $|C_E(x)| = 6 > 4$. It follows that $H$ is a $\{2,3\}$-group that does not contain any transpositions. Hence $H$ is of isomorphism type $C_2$, $E_4$, $C_4$, $\Sym_3$, or $\Alt_4$, as desired.

    Next, we suppose that $E \cong \SL_2(5)$ and let $\bar E := E/Z(E) \cong \Alt_5$. By Corollary 2.24 in \cite{Isaacs}, $|C_{\bar E}(\bar x)| \leq |C_E(x)|$
    for all $x \in G_\alpha$. Thus we can use the previous paragraph to see that $G_\alpha$ is isomorphic to one of the five groups listed there. If $x \in G_\alpha$ has order $3$, then $Z(E) \leq C_{E}(x)$ and $|C_{\bar E}(\bar x)|$ must be divisible by $3$ because $3 \in \pi(\bar{E})$. Then coprime action gives that $|C_E(x)| \geq 6$, which contradicts Lemma \ref{normaliser}\,(ii). Hence $G_\alpha$ is a $2$-group and the groups $C_2$, $E_4$ as well as $C_4$ are the only remaining possibilities.
\end{proof}

\begin{ex}\label{exRegComp}
    We show that all the cases in Lemma \ref{regularGaoptions} actually occur. Let $E := \Alt_5$ and let $H \leq \Sym_5$ be one of the following groups:
    
    $\< (12)(34) \> \cong C_2$, $\< (12)(34), (13)(24) \> \cong E_4$, $\< (1234) \> \cong C_4$, $\< (123),(12)(45) \> \cong \Sym_3$ or\\ $\< (123), (234) \> \cong \Alt_4$. 
    
    In every case, a small calculation shows that $\max \{ |C_E(\varphi)| \mid \varphi \in H^\#\} = 4$. According to Lemma \ref{RegNSG}, $G := E:H$ acts with fixity $4$ on $\Omega := G / H$ such that $E = E(G)$ acts regularly.

    In the same way, we can construct the examples for $E := \SL_2(5)$ by choosing $H$ as a non-trivial $2$-subgroup of $\PSL_2(5)$ or by choosing $H$ to be generated by an element of order $4$ in $\PGL_2(5)$.
\end{ex}

For the remainder of our analysis, we narrow down the possibilities for the action of $E(G)$ on one of its orbits, in particular the possibilities for the fixity.
Then, in each case, we can refer to earlier work and write down all cases exactly. This will lead to all possible structures of $F^*(G)$, and then we can prove our main theorem in the next section.

\subsection{Components that act with fixity 2 on their orbits}

\begin{lemma}\label{Efix2orbits}
    Suppose that Hypothesis \ref{hyp} holds and that $E(G)$ acts with fixity $2$ on $\alpha^{E(G)}$. Then $E(G)$ has at most two orbits on $\Omega$.
\end{lemma}
\begin{proof}
    Let $E := E(G)$. Since $E$ acts with fixity $2$, Lemma \ref{quasisimpleFix23} and Proposition \ref{onecomp} imply that $E$ is simple. According to our hypothesis, there exists some $x \in E_\alpha$ of prime order $p$ such that $|\fix_{\alpha^E}(x)| = 2$. It follows that $|\alpha^E| \equiv 2$ modulo $p$, and by Lemma \ref{normalaction}, every $\beta \in \Omega$ satisfies $|\beta^E| = |\alpha^E|$. If $p \geq 3$, then the congruence implies that $x$ has at least two fixed points on every $E$-orbit. If $p = 2$, then $x$ must also have at least two fixed points on every $E$-orbit: Otherwise it would induce an odd permutation on one of the orbits, and then $E$ would have a subgroup of index $2$. This is impossible. Therefore, the fixity $4$ hypothesis implies that there are at most two $E$-orbits.
\end{proof}

\begin{lemma}\label{EtransFix2}
	Let $G$ be a group that acts transitively, faithfully and with fixity at most $4$ on $\Omega$. Suppose that $E(G) \neq 1$ acts transitively with fixity $2$ on $\Omega$ and that $F(G) \neq 1$. Then $|F(G)| = 2$, $F^*(G)$ acts with fixity $4$, every involution in $F^*(G) \setminus (F(G) \cup E(G))$ has four fixed points, and one of the following holds:

	\begin{enumerate}
		\item $G \cong \Alt_5 \times C_2$ and $G_\alpha$ is isomorphic to $E_4$, $\Sym_3$ or $D_{10}$.
		\item $G \cong \Alt_5 : C_4$ and $G_\alpha \cong C_5 : C_4$.
		\item $G \cong \PSL_2(7) \times C_2$ and $G_\alpha \cong \Sym_4$.
	\end{enumerate}
\end{lemma}
\begin{proof}
	Let $E := E(G)$ and $F := F(G) \neq 1$. Since $E$ acts with fixity $2$, Lemma \ref{quasisimpleFix23} and Proposition \ref{onecomp} imply that $E$ is simple. In particular, $F^*(G) = E \times F$. If $F_\alpha \neq 1$ for some $\alpha \in \Omega$, then the facts that $E\leq N_G(F_\alpha)$ and $E \cap  N_{G_\alpha}(F_\alpha) = E_\alpha$ and  Lemma \ref{normaliser}\,(i) yield $|E:E_\alpha| \leq |N_{G}(F_\alpha): N_{G_\alpha}(F_\alpha)|
    \leq 2$. Since $E$ is simple, this is only possible if $E \leq G_\alpha$, which contradicts the faithful action of $G$. Thus, $F$ acts semi-regularly. Let $x \in E^\#$ be such that $|\FO(x)| = 2$. Then $F \leq C_G(x)$ and thus $F$ permutes the two fixed points of $x$. Since $F$ acts semi-regularly, this implies that $|F| = 2$ and thereby $F \le Z(G)$. Conversely, $Z(G) \le F$ and therefore $Z(G)=F$ has order $2$. 

	Consider the action of $\bar G := G / F$ on the set $\bar \Omega$ of $F$-orbits and let $\alpha \in \Omega$ and $\bar \alpha := \alpha^F \in \bar \Omega$. By Lemma \ref{centrereduction}, $\bar E = EF/F$ acts non-regularly and faithfully with fixity at most $4$ on $\bar \Omega$.

    Now we note two facts: 
    The first one is that $\bar E=EF/F \cong E/E \cap F \cong E$, in particular $E \cong \bar E=\overline{EF}$, because $EF=E \times F$.
    The second one is that $\alpha^E=\Omega$ by hypothesis, in particular $\alpha^E=\alpha^{EF}$. Since $|EF|=2 \cdot |E|$, this implies that $2 \cdot |E_\alpha|=|(EF)_\alpha|$.

    It follows from this and the first fact that $(EF)_\alpha \cong \bar{E}_{\bar \alpha}$ contains $E_\alpha$ as a subgroup of index $2$. Since $\bar E \cong E$ is simple, $\bar E$ must act with fixity $2$, $3$, or $4$ by Remark \ref{CompInd}~(b). We now go through Table \ref{TableAllFix} and look for fixity $2$ actions of $E$ with the following conditions: There exists a fixity $2$, $3$ or $4$ action for $\bar E$ such that $\bar{E}_{\bar \alpha}$ contains an index $2$ subgroup isomorphic to $E_\alpha$. We find that this is only possible if $E$ is isomorphic to $\Alt_5$, $\PSL_2(7)$, $\PSL_2(8)$ or $\PSL_2(9)$ and $(EF)_\alpha \cong \bar E_{\bar \alpha}$ is isomorphic to $\Sym_3$, $E_4$, $\Sym_4$, $D_{10}$, $D_{14}$ or $D_{18}$.

	In all these cases, there exists an involution $t \in (EF)_\alpha \setminus E_\alpha$. Let $t_E \in E$ and $t_F \in F$ be such that $t = t_E t_F$. Then $t_E$ and $t_F$ are involutions and $C_{EF}(t) = C_E(t_E) \times C_F(t_F)$. All possibilities for $E$ have just a single conjugacy class of involutions and $C_E(t_E) \in \syl_2(E)$. In particular, $C_{EF}(t)$ has order $8$ if $E \cong \Alt_5$ and order $16$ in the other three cases. Lemma \ref{normaliser}\,(ii) states that $|C_{EF}(t):C_{(EF)_\alpha}(t)| \leq 4$. The only possible combinations of $EF$ and $(EF)_\alpha$ with these properties are the ones listed in Case (a) and (c) of the statement. In particular $E \cong \Alt_5$ or $E \cong \PSL_2(7)$, which we will refer to later.
  
	Assume for a contradiction that $EF$ does not act with fixity $4$ on $\Omega$. Since $Z(EF) = F$ has order $2$, Lemma \ref{normaliser} implies that this is only possible if $EF$ acts with fixity $2$. According to Lemma \ref{centrereduction} and Remark \ref{CompInd}~(b), $\bar E = EF/F$ must therefore also act with fixity $2$ on $\bar \Omega$. But if $\bar E \cong \Alt_5$ and $\bar{E}_{\bar \alpha} \cong E_4$ or if $\bar E \cong \PSL_2(7)$ and $\bar E_{\bar \alpha} \cong \Sym_4$, then $\bar E$ acts with fixity $3$ by Table \ref{TableAllFix}. Otherwise $\bar E \cong \Alt_5$ and $\bar{E}_{\bar \alpha} \cong (EF)_\alpha$ is isomorphic to $\Sym_3$ or to $D_{10}$ and therefore $|C_{EF}(t) : C_{(EF)_\alpha}(t)| = 4$. This contradicts Lemma \ref{normaliser}\,(ii). Thus $EF$ must act with fixity $4$. Since there is only one class of involutions in $E$, and hence in $EF \setminus (F \cup E)$, and since some element of prime order in $EF \setminus E$ must have four fixed points, we conclude that every involution in $EF \setminus (F \cup E)$ has four fixed points.

    This proves all general statements of the lemma, and if $G=F^*(G)$, then we have already seen that (a) or (c) holds.
	It remains to show that Case (b) holds if $G \neq F^*(G)$. 
   
    Since $\aut(EF) = \aut(E)$ and $Z(EF) = Z(G) = F$, it follows that $\bar G$ is almost simple, but not simple. We stated the possibilities for $E(G)$ above, namely 
   $\Alt_5$ or $\PSL_2(7)$, and hence we know that $\bar G$ is isomorphic to $\Sym_5$ or to $\PGL_2(7)$. We recall that $E$ acts transitively on $\Omega$ by hypothesis. Then a Frattini argument gives that $G= EG_\alpha$. Therefore, $G/E = EG_\alpha/E \cong G_\alpha/E_\alpha$, and then
   we apply
    Lemma \ref{centrereduction}: With the notation there, $k=4$ and $|Z|=2$ (where $Z=Z(G)=F$) and $|\Omega|>8$. Hence the hypotheses are satisfied and Lemma \ref{centrereduction}~(i) gives that $G_\alpha \cong \bar{G}_{\bar \alpha}$ is isomorphic to a subgroup of $\bar G$. 
    
    If $E \cong \PSL_2(7)$, then $E_\alpha \cong \Alt_4$ has order $12$ and thus $G_\alpha \cong \bar{G}_{\alpha}$ has order $48$. But $\bar G \cong \PGL_2(7)$ does not contain a subgroup of order $48$. Thus $E \cong \Alt_5$, $E_\alpha$ is cyclic of order $2$, $3$ or $5$ and $\bar{E}_{\bar \alpha}$ is isomorphic to one of the groups listed in Case (a). Since $\bar{G} \cong \Sym_5$ and $|\bar{G}_{\bar \alpha} : \bar{E}_{\bar \alpha}| = 2$, it follows that $G_\alpha \cong \bar{G}_{\bar \alpha}$ is isomorphic to the normaliser of a Sylow $2$-, $3$-, or $5$-subgroup of $\Sym_5$. In the first two cases, $G / E \cong G_\alpha / E_\alpha$ is elementary abelian. Since $|F| = 2$, we then obtain that $G \cong \Sym_5 \times C_2$. Let $G_0 \trianglelefteq G$ be such that $G_0 \cong \Sym_5$ and $G = G_0 \times F$ and let $\pi : G \to G_0$ denote the projection on $G_0$. Since $F_\alpha = 1$, we have that $G_\alpha^\pi \cong G_\alpha$. Let $s \in G_\alpha$ be such that $s^\pi$ corresponds to a transposition in $G_0 \cong \Sym_5$. If $G_\alpha \cong \Sym_3 \times C_2$, then we may choose $s$ such that $s$ inverts both elements of order $3$ under conjugation. In any case, $C_G(s)$ contains an element of order $3$ which is not contained in $G_\alpha$. It follows that $|C_G(s) : C_{G_\alpha}(s)|$ is divisible by $3$. Moreover, $F \leq C_G(s)$ has order $2$ and acts semi-regularly on $\Omega$. Thus $|C_G(s) : C_{G_\alpha}(s)|$ is also divisible by $2$, hence by $6$, and this contradicts Lemma \ref{normaliser}\,(ii).
	The only remaining case is $G_\alpha \cong C_5 : C_4$. Here $G/E \cong G_\alpha / E_\alpha \cong C_4$, as described in Case (b).
\end{proof}

\begin{ex}\label{EtransFix2eg}
	We show that all cases of Lemma \ref{EtransFix2} actually occur. To construct examples for Case (a) where point stabilisers are isomorphic to $\Sym_3$ or to $D_{10}$ or for Case (b), we can simply apply Lemma \ref{centrereduction} to the groups described in Example \ref{ExFinalTheorem2eiv}.

    The action of the full automorphism group of an icosahedron on the set of edges is an example for Case (a) where the point stabilisers are isomorphic to $E_4$.

    For Case (c), we turn to the Transitive Groups Library~\cite{TransGrp} of \texttt{GAP} and look at the group \texttt{G := TransitiveGroup(14, 17)}. Using \texttt{GAP}, we see that $G \cong \PSL_2(7) \times C_2$ and that the point stabiliser $G_1$ is isomorphic to $\Sym_4$. Moreover, \texttt{GetFixity(G)} evaluates to $4$ while \texttt{GetFixity(DerivedSubgroup(G))} returns $2$. Hence $G$ is an example for Case (c).
\end{ex}

\begin{lemma}\label{Efix2}
Suppose that Hypothesis \ref{hyp} holds and that $E(G) \neq 1$ acts with fixity $2$ on $\alpha^{E(G)}$. Then $|F(G)| \leq 2$. More precisely, one of the following holds:
\begin{enumerate}[(1)]
    \item $F(G) = 1$ and $G$ is almost simple.
    \item $|F(G)| = 2$ and $G = G_0 \times Z(G)$ for an almost simple group $G_0$. Moreover, $G_0$ has two orbits on which it acts with fixity $2$, and $Z(G)$ interchanges these two orbits.
	\item $|F(G)| = 2$, $E(G)$ acts transitively on $\Omega$, and one of the cases in Lemma \ref{EtransFix2} holds.
\end{enumerate}
\end{lemma}
\begin{proof}
	Let $F := F(G)$ and $E := E(G)$. Note that $E$ acts faithfully on $\alpha^E$ by Remark \ref{CompInd}~(a), and $F$ acts semi-regularly by Lemma \ref{FfixedE1}. Applying Lemma \ref{quasisimpleFix23} and Proposition \ref{onecomp} to the action of $E$ on $\alpha^E$ shows that $E$ must be simple. It follows that $F^*(G) = F \times E$. If $F = 1$, then $F^*(G) = E$ is simple and $G$ is almost simple. This is Case (1). We may therefore suppose that $F \neq 1$.

	If $E$ acts transitively on $\Omega$, then Lemma \ref{EtransFix2} yields that $|F| = 2$ and that one of the cases listed in the Lemma holds. This is (3). 

Thus we may suppose that $E$ does not act transitively. Lemma \ref{Efix2orbits} states that $E$ has exactly two orbits. Let $\beta \in \Omega \setminus \alpha^{E}$ and let $K$ be the kernel of the action of $G$ the set $\{ \alpha^E, \beta^E \}$. Then $|G:K| = 2$ and $F^*(K) = F^*(G) \cap K = (F \cap K) \times E$. Assume for a contradiction that $F \cap K \neq 1$. Then Lemma \ref{EtransFix2}, applied to the action $K$ on $\alpha^E$ and $\beta^E$, gives that every involution in $(K \cap FE) \setminus (F \cup E)$ has four fixed points on $\alpha^E$ and $\beta^E$. This makes eight fixed points on $\Omega$, which is impossible by our fixity $4$ hypothesis. Hence $F \cap K = 1$ and $K = E$. 

This implies that  $|F| = 2$ and that 
$F$ interchanges the two orbits of $E$ on $\Omega$. Therefore, the actions of $E$ on $\alpha^E$ and on $\beta^E$ are isomorphic, see Lemma~\ref{normalaction},
and we obtain Case (2).
\end{proof}

\subsection{Components that act with fixity 3 or 4 on their orbits}

\begin{lemma}\label{Efix3}
	Suppose that Hypothesis \ref{hyp} holds and that $E := E(G) \neq 1$ acts with fixity $3$ on $\alpha^{E}$. Then $E \cong \Alt_6$ and one of the following is true:
\begin{enumerate}[(1)]
    \item If $|\Omega|=6$, then $G$ is isomorphic to $\Sym_6$.
    \item If $|\Omega|=12$, then $G$ is isomorphic to $\M_{10}$, $\PGL_2(9)$, or $\aut(\Alt_6)$.
\end{enumerate}
\end{lemma}
\begin{proof}
	By Lemma~\ref{quasisimpleFix23}, the quasi-simple group $E$ is in fact simple. Lemma~\ref{normalaction} tells us that all $E$-orbits in $\Omega$ have
	the same size.
	
	Let $a\in E_{\alpha}$ be of prime order $p$ and such that it fixes exactly three points in~$\alpha^E$. Then
	$|\alpha^E|\equiv 3 \mod p$.
	We will make a case distinction depending on whether $\alpha^E=\Omega$ or not.
	
	First suppose that $\alpha^E\neq \Omega$.
	Then $a$ can fix up to four points in $\Omega$ because $G$ acts with fixity $4$.
	However, we will argue now that $a$ fixes exactly three points in~$\Omega$.
	For a contradiction, assume that $a$ fixes $\delta\in \Omega\setminus \alpha^E$.
	Then $\fix_{\Omega \setminus \alpha^E}(a)=\{\delta\}$, and hence
	$1\equiv |\delta^E|=|\alpha^E|\equiv 3\mod p$.
	Therefore, $p=2$ and $a$ acts as a product of transpositions on both $E$-orbits. Since $a$ has three fixed points on $\alpha^E$ and one fixed point on $\delta^E$, it must therefore act as an odd permutation on one of the orbits. It follows that $E$ has normal subgroup of index $2$ in contradiction to $E$ being non-abelian simple.\\
	As a consequence, $a$ fixes exactly three points in $\Omega$. We
	still suppose that $\alpha^E\neq \Omega$. Thus, there exists a point $\delta\in \Omega\setminus \alpha^E$ and we see that
	$0\equiv |\delta^E|=|\alpha^E|\equiv 3\mod p$. Therefore, $p=3$ and $|E_{\alpha}|$ and $|\alpha^E|$ are both divisible by $3$. By Table \ref{TableAllFix}, the group $E$ is isomorphic to $\Alt_6$ and $|\alpha^E|\in\{6,15\}$ or to $\Alt_7$ and
	$|\alpha^E|=15$.
	If $|\alpha^E| = |\delta^E| = 15$, then every involution $t \in E$ must fix three points in both $E$-orbits, because otherwise it induces an odd permutation. In total, $t$ has at least six fixed points in contradiction to our main hypothesis.\\
	As a consequence, $|\alpha^E|=6$ and $E\cong \Alt_6$ with point stabilisers isomorphic to $\Alt_5$.
	Again, every involution $t \in E$ induces an even permutation and must therefore fix two points in every $E$-orbit. Hence the fixity $4$ hypothesis yields that there are two $E$-orbits and that $|\Omega| = 2 \cdot |\alpha^E| = 12$. If $E_\alpha$ is conjugate to $E_\delta$ in $E$, then the action of $E$ on $\alpha^E$ is equivalent to the action of $E$ on $\delta^E$, and every element in $E$ with three fixed points on $\alpha^E$ must also have three fixed points on $\delta^E$. This is a contradiction to the fixity $4$ hypothesis. Hence $E_\alpha$ and $E_\delta$ are representatives for the two distinct conjugacy classes of $E \cong \Alt_6$ which are isomorphic to $\Alt_5$.\\

    Let $x \in G$ such that $x$ interchanges the two $E$-orbits. Then conjugation by $x$ interchanges the two conjugacy classes which are isomorphic to $\Alt_5$, and it follows that $x$ induces a non-trivial automorphism on $\Alt_6$. Since $[F(G), E] = 1$, it follows that $F(G)$ stabilises both $E$-orbits. Since $|\alpha^E| = |\alpha^{F^*(G)}| = 6 > 4$, $F^*(G)$ acts faithfully on $\alpha^E$ and is therefore isomorphic to a subgroup of $\Sym_6$. This is only possible if $F(G) = 1$ and $F^*(G) = E \cong \Alt_6$. Now $G$ is isomorphic to a subgroup of $\aut(\Alt_6)$ which contains an automorphisms that swaps the two conjugacy classes of subgroups isomorphic to $\Alt_5$. These possibilities are listed in Case (2).

	Now, for the second possibility, suppose that $\alpha^E=\Omega$, which means that $E$ acts transitively on $\Omega$.\\
	For a contradiction, assume that $E_\alpha$ has order coprime to $6$. Let $x \in G_\alpha$ be of prime order such that $|\FO(x)| = 4$. Then $x \notin E$ and every $1 \neq y \in C_E(x)$ stabilises $\FO(x)$ set-wise. Since $E$ acts with fixity $3$, the element $y$ induces a non-trivial permutation on $\FO(x)$ and must therefore have order divisible by $2$ or $3$. The assumption $(|E_\alpha|, 6) = 1$ now yields $C_{E_\omega}(x) = 1$ for all $\omega \in \Omega$. Hence $C_E(x)$ acts semi-regularly on the $4$-set $\FO(x)$ and must therefore be a $2$-group. Now the theorem in \cite{Fu1984} implies that $x$ is an involution and Lemma \ref{ceninvsimple} gives $E \cong \Alt_5$. But the only fixity $3$ action of $\Alt_5$ has point stabilisers of even order. Hence $E_\alpha$ cannot have order coprime to $6$.\\
	According to Table \ref{TableAllFix}, $E$ must be isomorphic to $\Alt_5$, $\Alt_6$, $\Alt_7$, $\PSL_2(7)$ or $\M_{11}$ and $|\Omega| \in \{ 6,7,11,15 \}$. Assume for a contradiction that $|\Omega| \in \{7,11,15\}$ and let $x \in G_\alpha$ be of prime order such that $|\FO(x)| = 4$. Then $x$ is of order $3 = 7-4$, $7 = 11-4$ or $11=15-4$, respectively. Since the outer automorphism group of all possibilities for $E$ is a $2$-group, we conclude that $x$ induces an inner automorphism. Now $E\<x\> \neq E$ acts with fixity $4$ on $\Omega = \alpha^E$ and $|Z(E \<x\>)| = o(x)$ is coprime to $4$ in contradiction to Lemma \ref{normaliser}. Hence $|\Omega| = 6$, in which case $E \cong \Alt_6$ and $E_\alpha \cong \Alt_5$. Since the action of $G$ is faithful, we can identify $G$ with a subgroup of $\Sym_{\Omega} \cong \Sym_6$ which contains a proper subgroup isomorphic to $\Alt_6$. This is only possible if $G \cong \Sym_6$ as in Case (1).
\end{proof}

\begin{lemma}\label{compfix4FGsimple}
	Suppose that Hypothesis \ref{hyp} holds, that $E(G)$ is simple and that it acts with fixity $4$ on $\alpha^{E(G)}$.
	If $F(G) \neq 1$, then $|F(G)| = 2$ and there exists an odd prime power $q \geq 7$ such that $E(G) \cong \PSL_2(q)$. More precisely, in the latter case one of the following holds:
	\begin{enumerate}[(1)]
		\item $q \equiv 1$ modulo $4$ and $E(G) \cap G_\alpha$ is of isomorphism type $C_{\frac{q-1}{4}}$ or $E_q : C_{\frac{q-1}{4}}$.
		\item $q \equiv -1$ modulo $4$ and $E(G) \cap G_\alpha$ is of isomorphism type $C_{\frac{q+1}{4}}$.
	\end{enumerate}
\end{lemma}
\begin{proof}
Let $E := E(G)$ and suppose $F := F(G) \neq 1$. By Lemma \ref{FfixedE1}, $E \neq 1$ implies that $F$ acts semi-regularly on $\Omega$. Since $E$ acts with fixity $4$ on $\alpha^E$, we may assume that $G = EF$ and that $\Omega = \alpha^{EF}$. Let $x \in E$ such that $x$ fixes exactly four points in $\alpha^E$. Since $F$ centralises $x$, it stabilises $\fix_{\Omega}(x) \subseteq \alpha^E$ and therefore also $\alpha^E$. In particular, we have $\alpha^E = \alpha^{EF} = \Omega$ and $E$ acts transitively on $\Omega$. Moreover, the semi-regular action of $F$ yields $|F| \in \{2,4\}$. Hence $F$ is abelian and therefore $F = Z(EF) = Z(G)$. Let $F_0 \leq F$ be a subgroup of order $2$, $\bar G := G / F_0$ and $\bar \Omega$ be the set of $F_0$-orbits.

First, we consider the group $H := E \times F_0$. Note that $H$ acts transitively and with fixity $4$ on $\Omega$ because $E$ acts with fixity $4$ on $\alpha^E = \Omega$. We apply Lemma \ref{centrereduction} to this action and obtain that the simple group $\bar H \cong E$ acts transitively, faithfully and non-regularly with fixity at most $4$ on $\bar \Omega$. Moreover, $H_\alpha \cong \bar{H}_{\bar \alpha} = \bar{E}_{\bar \alpha}$ contains $E_\alpha$ as a subgroup of index $2$. By Remark \ref{CompInd}~(b), the action of $\bar H$ has fixity $2$, $3$ or $4$. We go through Table \ref{TableAllFix} and look for fixity $4$ actions of $E$ such that there exists a fixity $2$, $3$ or $4$ action for $\bar H \cong E$ such that $\bar{H}_{\bar \alpha} \cong H_\alpha$ contains $E_\alpha$ as an index $2$ subgroup. We find that this is only possible if $E \cong \PSL_2(q)$ for some odd prime power $q \geq 7$ with $E_\alpha$ as described in Case (1) or (2), or if $E \cong \Alt_6 \cong \PSL_2(9)$ and $E_\alpha$ is elementary abelian of order $9$. In the latter case, $H_\alpha \cong E_9 : C_2$ and since $\PSL_2(9)$ contains a single class of involutions and $H = F_0 \times E$, the involutions in $H_\alpha$ must be $2$-central. In particular, $|C_H(E)|$ is divisible by $16$ in contradiction to Lemma \ref{normaliser}\,(ii). Thus, $E_\alpha$ must be as described in Case (1) or (2). Note that in both cases, $\bar H \cong E$ acts with fixity $2$ on $\bar \Omega$.

It remains to show that $|F| = 4$ is impossible. Assume otherwise and consider the action of $\bar G = \bar F \times \bar E \cong C_2 \times E$ on $\bar \Omega$. By Lemma \ref{centrereduction}, this action has fixity at most $4$ and we already saw that $\bar E$ acts transitively and with fixity $2$ on $\bar \Omega$ and that $\bar E \cong E \cong \PSL_2(q)$ for some odd prime power $q \geq 7$. Applying Lemma \ref{Efix2orbits} to the action of $\bar F \times \bar E$ on $\bar \Omega$ gives $q = 7$ and $\bar{E}_{\bar \alpha} \cong \Alt_4$. But $E_{\alpha} \cong C_2$ as described in Case (2) and thus $|\bar{E}_{\bar \alpha}| = 4$. This contradiction show that $|F| = 4$ is impossible.
\end{proof}

\bigskip
\begin{lemma}\label{compfix4FGquasisimple}
	Suppose that Hypothesis \ref{hyp} holds, that $E(G)$ is quasi-simple, but not simple, and that $E(G)$ acts with fixity $4$ on $\alpha^{E(G)}$. Then one of the following holds:
	\begin{enumerate}[(1)]
		\item $F(G) = Z(E(G)) \cong C_2$.
		\item $G \cong \SL_2(5) : C_2$ with $G_\alpha \cong \Sym_3$ or $G_\alpha \cong D_{10}$.
		\item $G \cong \GL_2(5)$ with $G_\alpha \cong C_5 : C_4$.
	\end{enumerate}
\end{lemma}
\begin{proof}
Let $E := E(G)$, $F := F(G)$ and $\Delta := \alpha^{EF}$. Let $x \in E_\alpha$ such that $x$ has four fixed points on $\alpha^E$. Then $\FO(x) \subseteq \alpha^E$ and since $F \leq C_E(x)$, it must stabilise $\FO(x)$ and thereby $\alpha^E$ set-wise. Thus $\Delta = \alpha^{EF} = \alpha^E$ and $E$ acts transitively on $\Delta$.

 We know by Lemma \ref{quasisimple} that $Z := Z(E) = F \cap E$ has order $2$ and that $E/Z$ acts with fixity $2$ on the set of $Z$-orbits on $\Delta$. If $F = Z$, then $|F| = 2$ as described in Case (1). Assume that $|F| > 2$. Since $Z \leq Z(G)$, we can use Lemma \ref{centrereduction} to see that $\bar G := G / Z$ acts with fixity at most $4$ on the set of $Z$-orbits $\bar \Omega$ and that $G_\alpha \cong \bar{G}_{\bar \alpha}$. Since $F^*(\bar G) = \bar F \times \bar E$ and $\bar E$ acts with fixity $2$ on $\bar \Delta$, Lemma \ref{EtransFix2} yields that $\bar E$ is isomorphic to $\Alt_5$ or $\PSL_2(7)$ and thus $E$ is isomorphic to $\SL_2(5)$ or $\SL_2(7)$. Moreover, it states that $F^*(\bar G)$ acts with fixity $4$ on $\Delta$ and thus $\bar G$ acts with fixity $4$ on $\bar \Omega$. This allows us to apply Lemma \ref{Efix2} to the action of $\bar G$ and since $\bar{\alpha}^{\bar E \bar F} = \bar \Delta$, we obtain that $\Delta = \Omega$. Thus $E$ is transitive on $\Omega$ and we can once again turn to Lemma \ref{EtransFix2}. Comparing the possible point stabiliser structures listed there with the point stabilisers described in Lemma \ref{SLquasisimple} leads to the Cases (2) and (3) of the statement.
\end{proof}


\section{The proof of Theorem \ref{FinalTheorem}}

We have already seen that all cases in the theorem occur, which is why we focus on the converse statement now. 

As stated in our main result, we work under Hypothesis \ref{hyp}.
We first suppose that $F(G)\cap G_{\alpha} \neq 1$. Then $E(G)=1$ by Lemma~\ref{FfixedE1}. If $F(G)$ acts faithfully on $\alpha^{F(G)}$, then Lemma \ref{FixityFG} states that $F(G)$ is a $2$-group with sectional $2$-rank at most $4$ that acts with fixity $4$ on $\alpha^{F(G)}$ and by Lemma \ref{2^3.3^2}, $|G_\alpha|$ divides $2^3 \cdot 3^2$. This is Case (1)\,(a).
Otherwise, we apply Lemma~\ref{FittingGAP} and find that all the information about $F(G)$ is contained in Case (1)\,(b). %
	
	\medskip
	
	Therefore, from now on, we suppose that $F(G)$ acts semi-regularly on $\alpha^{F(G)}$.
	If $F(G)\neq O_2(G) \times O_3(G)$, then Lemma~\ref{FGsemireg} gives that, for all $p\in \pi(G_{\alpha})$, the Sylow
	$p$-subgroups of $G$ have rank $1$. Thus, it remains to prove that one of the cases (a) -- (e) of~(2) holds.

	Suppose that Case (2)\,(a) is not true. Then Proposition~\ref{onecomp} shows that $E(G)$ is quasi-simple. By Hypothesis \ref{hyp}, $E(G)$ acts with fixity at most $4$ on $\alpha^{E(G)}$ and Remark \ref{CompInd}~(b) states that fixity $1$ is impossible. We are thus left with four different cases.

    \begin{itemize}
        \item If $E(G)$ acts semi-regularly on $\Omega$, then Lemma~\ref{Enonreg} gives $E(G)/Z(E(G)) \cong \Alt_5$ and Lemma \ref{regularGaoptions} lists all possibilities for the point stabilisers. This is Case (2)~(b).

        \item If $E(G)$ acts with fixity $2$ on $\alpha^{E(G)}$, then $E(G) \cong \PSL_2(q), Sz(q)$ or $\PSL_3(4)$ by Lemma \ref{quasisimpleFix23} and Table \ref{TableAllFix}. We apply Lemma~\ref{Efix2} and see that the cases listed there correspond to the subcases in Case (2)\,(c) of Theorem \ref{FinalTheorem}.

        \item If $E(G)$ acts with fixity $3$ on $\alpha^{E(G)}$, then Lemma~\ref{Efix3} gives all information in Case~(2)\,(d) except for the size of the Fitting subgroup of $G$. But since we know what the group $G$ is in all cases, we can see that $F(G)=1$.

        \item If $E(G)$ acts with fixity $4$ on $\alpha^{E(G)}$, then we have to distinguish between $Z(E(G)) = 1$ and $Z(E(G)) \neq 1$.
		If $E(G)$ is simple, then Theorem~1.3 in \cite{BHMSW} shows that $E(G)$ is one of the groups described in (2)\,(e)\,(i). Moreover, if $|F(G)| > 1$, then $|F(G)| = 2$ by Lemma~\ref{compfix4FGsimple} and we obtain Case (2)\,(e)\,(ii).
		Otherwise $E(G)$ is quasi-simple and non-simple. Then we apply Lemma~\ref{quasisimple} and we see that $E(G)$ is one of the groups described in (2)\,(e)\,(iii). Finally, Lemma~\ref{compfix4FGquasisimple} shows that Case (2)\,(e)\,(iv) holds whenever $|F(G)| \neq 2$.
    \end{itemize}

\normalem

\end{document}